\documentclass[a4, 12pt]{article}
\usepackage{amsmath, amssymb, amscd, amsthm}
\usepackage{graphicx}

\makeatletter
\@addtoreset{equation}{section}

\makeatother

\def\ii{{\rm{i}}}

\def\ee{{\rm{e}}}
\def\uu{{{u}}}

\def\tr{{\rm{tr}}}
\def\prop{{\rm{prop}}}

\def\EE{{\bf{E}}}
\def\RR{{\mathbb{R}}}
\def\EE{{\mathbb{E}}}

\def\cE{{\mathcal{E}}}
\def\cF{{\mathcal{F}}}
\def\cS{{\mathcal{S}}}
\def\cA{{\mathcal{A}}}
\def\cL{{\mathcal{L}}}
\def\ZZ{{\mathbb{Z}}}
\def\uu{{{u}}}
\def\KK{{{K}}}

\def\supp{{\rm{supp}}}
\def\SDiff{{\rm{SDiff}}}
\def\IFluid{{\rm{IFluid}}}
\def\Coor{{\rm{Coor}}}
\def\Diff{{\rm{Diff}}}
\def\cC{{\mathcal{C}}}
\def\sdiff{{\mathfrak{sdiff}}}
\def\diff{{\mathfrak{diff}}}
\def\epsT{{\varepsilon_T}}
\def\epsX{{\epsilon_\xi}}
\def\old{{\mathrm{old}}}
\def\new{{\mathrm{new}}}
\def\free{{\mathrm{free}}}
\def\dist{{\mathrm{dist}}}
\def\exact{{\mathrm{exact}}}
\def\sing{{\mathrm{sing}}}
\def\Area{{\mathrm{Area}}}
\def\Length{{\mathrm{Length}}}
\def\two{{\mathrm{two}}}
\def\mul{{\mathrm{mul}}}
\def\tri{{\mathrm{tri}}}
\def\wall{{\mathrm{wall}}}
\def\tgamma{{\tilde{\gamma}}}
\def\um{{\mu\mathrm{m}}}
\newtheorem{theorem}{Theorem}
\newtheorem{lemma}{Lemma}

\newtheorem{assumption}{Assumptions}
\newtheorem{definition}{Definition}
\newtheorem{proposition}{Proposition}
\newtheorem{myremark}{Remark}

\newcommand{\by}[1]{{\sc{#1}, }}
\newcommand{\jour}[1]{{\textrm{#1} }}
\newcommand{\Paper}[1]{{\em{#1}, }}
\newcommand{\vol}[1]{{\textrm{#1}} }
\newcommand{\yr}[1]{{(\textrm{#1})} }
\newcommand{\pages}[1]{{pp.\textrm{#1}}}
\newcommand{\Book}[1]{{\em{#1},} }
\newcommand{\publ}[1]{{(\textrm{#1}, }}
\newcommand{\publaddr}[1]{{\textrm{#1}, }}
\newcommand{\byr}[1]{{\textrm{#1}) }}

\begin{document}

\title{Surface tension of multi-phase flow with multiple junctions
      governed by the variational principle  }


\author{Shigeki Matsutani 
\and Kota Nakano \and Katsuhiko Shinjo}




\maketitle

\begin{abstract}
We explore a computational model of an incompressible fluid
with a multi-phase field in three-dimensional Euclidean space.
By investigating an incompressible fluid with a two-phase field
geometrically,
we reformulate the expression of the surface tension
for the two-phase field found by 
Lafaurie, Nardone, Scardovelli, Zaleski and Zanetti
(J. Comp. Phys. \vol{113} \yr{1994} \pages{134-147})
as a variational problem related to an infinite dimensional Lie group,
the volume-preserving diffeomorphism.
The variational principle to the action integral
with the surface energy reproduces their Euler equation
of the two-phase field with the surface tension.
Since the surface energy of multiple interfaces even with singularities
is not difficult to be evaluated in general 
and the variational formulation
works for every action integral,
the new formulation enables us to
extend their expression to that of a multi-phase ($N$-phase, $N\ge2$) flow
and to obtain a novel Euler equation with the surface tension 
of the multi-phase field.
The obtained Euler equation governs 
the equation of motion of the multi-phase field
with different surface tension coefficients without 
any difficulties for the singularities at multiple junctions.
In other words, we unify the theory of
multi-phase fields which express low dimensional interface geometry
and the theory of the incompressible fluid dynamics on 
the infinite dimensional geometry as a variational problem.
We apply the equation to the contact angle  problems
at triple junctions.
We computed the fluid dynamics for a two-phase field with a wall
numerically and show the numerical computational results
that for given surface tension coefficients, 
the contact angles are generated
by the surface tension as results of balances of the kinematic energy
and the surface energy.
Keywords: 
multi-phase flow\and surface tension\and multiple junction\and 
volume-preserving diffeomorphism

37K65 \and 
58E12 \and 
76T30 \and 
76B45      

\end{abstract}


\section{Introduction} \label{sec:int}

Recently, since the developments of both hardware and software
in computer science enable us to simulate complex
physical processes numerically, such computer simulations
become more important from industrial viewpoints.
Especially 
the computation of the incompressible multi-phase fluid dynamics
has crucial roles in order to evaluate the behavior of several
devices and materials in a micro-region,
{\it {e.g.}}, ink-jet printers, solved toners and so on.
In the evaluation, it is strongly required that
the fluid interfaces with multiple junctions are 
stably and naturally computed from these practical reasons.

In this article, 
in order to handle the fluid interfaces with multiple junctions 
in a three dimensional micro-region, 
we investigate a surface tension of an incompressible multi-phase 
flow with multiple junctions as a numerical computational method
under the assumption that the Reynolds number is not so large.
In the investigation, we encounter many interesting mathematical
objects and results,
which are associated with
low dimensional interface geometry having singularities, and
with the infinite dimensional geometry of incompressible fluid dynamics.
Further since even in a macroscopic theory,
we introduce artificial intermediate regions in the material
interfaces among different fluids or among a solid and fluids,
the regions give a resolution of the singularities in the interfaces
to provide extended Euler equations naturally.
Thus even though we consider the multi-phase fluid model as
a computational model,
we believe that it must be connected with mathematical nature of 
real fluid phenomena as their description.
We will mention the background, the motivation and the strategy of
this study more precisely as follows.

For a couple of decades, 
in order to represent the physical process with
the interfaces of the multi-phase fluids, the computational
schemes have been studied well.
These schemes are mainly classified into two types.
The first type is based on the level-set method \cite{S}
discovered by H-K. Zhao, T. Chan, B. Merriman, S. Osher
and L. Wang \cite{ZCMO,ZMOW}.
The second one
is based on the phase-field theory, which was found by
 J. U. Brakbill, D. B. Kothe and C. Zemach \cite{BKZ},
and B. Lafaurie, C. Nardone, R. Scardovelli, S. Zaleski, and G. Zanetti
\cite{LZZ}.
The authors in Reference \cite{LZZ} called the scheme SURFER.
Following them, there are many studies on the SURFER scheme,
{\it{e.g.}}, \cite[references therein]{AMS,Cab,Jac}.

The level-set method is a computational method in which we describe a
(hyper-)surface in terms of zeros of the level-set function, {\it{i.e.}}, a
real function whose value is a signed distance from the surface,
such as $q(x)$ in Section \ref{sec:two-one}.
Using the scheme based upon the level-set method
in the three dimensional Euclidean space,
we can deal well with  topology changes,
geometrical objects with singularities, {\it{e.g.}}, cusps,
the multiple junctions of materials, and so on.
However in the computation, we need to deal with the
constraint conditions even for two-phase fluids
\cite{ZCMO,ZMOW}.
A dynamical problem with constraint conditions
is basically complicate and sometimes gives difficulties
to find its solution since the constraint conditions sometimes
generate an ill-posed problem in the optimization. 
In the numerical computation for incompressible fluid,
we must check the consistency between 
the incompressible condition and the constraint condition.
The check generally requires a complicate implementation
of the algorithm, and increases computational cost.
Its failure sometimes makes the computation unstable, especially
when we add some other physical conditions.
Since instability disturbs the evaluation of a complex
system as a model of a real device, it must be avoided.

On the other hand, 
using the SURFER scheme \cite{LZZ}, 
we can easily compute effects of the surface tension of a two-phase fluid
in the Navier-Stokes equation.
The phase field model is the model that we represent materials in terms
of supports of smooth functions which roughly correspond to
the partition of unity in pure mathematics \cite[I p.272]{KN}
as will be mentioned in Sections \ref{sec:four} and \ref{sec:five}.
We call these functions
{\lq\lq}color functions{\rq\rq} or {\lq\lq}phase fields{\rq\rq}.
The phase fields have artificial intermediate regions which
represent their interfaces  approximately.
In the SURFER scheme \cite{LZZ},
the surface tension is given as a kind of stress force, or volume force
due to the intermediate region.
Hence the scheme makes the numerical computations of the surface tension 
stable.
However
it is not known how to consider a multi-phase ($N$-phase, $N\ge2$) flow
 in their scheme.
In Reference \cite{BKZ}, the authors propose a method as an
extension of the SURFER scheme \cite{LZZ}
 to the contact angle problem
by imposing a constraint to fix its angle.
In this article, we will generalize 
 the SURFER scheme to multi-phase flow
without any constraints.

Nature must not impose any constraints even at such a triple junction,
which is governed by a physical principle. If it is a 
Hamiltonian system, its determination must obey the minimal
principle or the variational principle.
We wish to find a theoretical framework
in which we can consistently
handle the incompressible flows with interfaces including 
the surface tensions and the multiple junctions without any constraints.
As the multiple junctions should be treated as singularities in 
a mathematical framework which are very difficult to be handled
in general, 
it is hard to extend
mathematical approaches for fluid interface problems without
a multiple junction
\cite{BG,SZ} to a theory for the problem with multiple junctions.
Our purpose of this article is to find such a theoretical
framework which enables us to solve the fluid interface problems with 
multiple junctions numerically
as an extension of the SURFER scheme.

For the purpose, we employ the phase field model.
The thickness of the actual intermediate region 
in the interface between a solid and a fluid or between
two fluids is of atomic order and 
is basically negligible in the macroscopic theory. 
However the difference
between zero and {\lq\lq}the limit to zero{\rq\rq} 
sometimes brings a crucial difference in physics and
mathematics; for example, in the Sato hyperfunction theory,
the delta function is regarded as a function in the boundary
of the holomorphic functions \cite{KKK,II}, {\it{i.e.}},
$\displaystyle{\delta(x) =
 \lim_{\epsilon \to 0}\frac{1}{2\pi \ii}
\left(\frac{1}{x - \ii \epsilon}
-\frac{1}{x + \ii \epsilon}\right)}$
$\displaystyle{
\equiv \lim_{\epsilon \to 0}\frac{1}{\pi }
\frac{\epsilon}{x^2 +  \epsilon^2}}$.
As mentioned above,
the phase field model has the artificial intermediate region which is
controlled by a small parameter $\epsilon$ 
and appears explicitly even as a macroscopic
theory. We regard that it represents the effects coming from
the actual intermediate region of materials.
Namely, we regard that the stress force expression in 
the SURFER scheme
is caused by the artificial intermediate region of the phase-fields
and it represents well the surface effect
coming from that of real materials.

In order to extend the stress force expression of the two-phase
flow to that of the multi-phase ($N$-phase, $N\ge2$) flow,
we will first reformulate 
the SURFER scheme in the framework of the variational theory.
In Reference \cite{Jac}, a similar attempt was reported
but unfortunately there were not precise derivations. Our
investigations in Section \ref{sec:four}  
show that the surface tension expression of
the SURFER scheme
is derived as a momentum conservation in 
Noether's theorem \cite{BGG,IZ} and its derivation
 requires a generalization of the Laplace equation \cite{LL} as
the Euler-Lagrange equation \cite{AM,BGG}, which is not trivial
even for a static case.

In order to deal with this problem in a dynamics case consistently,
we should also consider the Euler equation in the framework
of the variational principle.
It is well-known that the incompressible fluid dynamics is
geometrically interpreted as 
a variational problem of an infinite
dimensional Lie group, related to diffeomorphism, due to
V. I. Arnold \cite{Ar,AK}, D. Ebin and J. Marsden \cite{EM}, H. Omori \cite{O}
and so on. Following them, there are so many related works
\cite{AF,B,K,NHK,Schm,Shk,Shn,V}.

On the reformulation of the SURFER scheme \cite{LZZ} for the dynamical case,
we introduce an action integral including 
the kinematic energy of the incompressible fluid and the surface energy.
The variational method reproduces the governing equation in
the SURFER scheme.

After then, we extend the surface energy to that of multi-phase fields 
and add the energy term to the action integral.
The variational principle of the action integral
leads us to a novel expression of the surface tension
and the extended Euler equation which we require.
Using the extended Euler equation,
we can deal with the surface tensions of the 
multi-phase flows, the multiple junctions of the 
of phase fields including singularities, the topology changes and so on.
We can also compute a wall effect naturally and a contact
angle problem.
The computation of the governing equation is freed from any
constraints, except the incompressible condition.

In other words, in this article, we completely unify
the theory of the multi-phase ($N$-phase, $N\ge2$) field
and the theory of the incompressible fluid dynamics of Euler equation
as an infinite dimensional geometrical problem.

Contents are as follows:
Section \ref{sec:two} is devoted to the preliminaries
 of the theory of surfaces
in our Euclidean space from a low-dimensional
differential geometrical viewpoint
\cite{M,GMO,FW} and Noether's theorem in the  
classical field theory \cite{AM,BGG,IZ}.
Section \ref{sec:three} reviews the derivation of the Euler equation
to the incompressible fluid dynamics following the variational method
for an infinite-dimensional Lie algebra based upon Reference \cite{EM}. 
In Section \ref{sec:four}, we reformulate the SURFER scheme \cite{LZZ}.
There the Laplace equation for the surface tension 
and the Euler equation in Reference \cite{LZZ}
are naturally obtained by the variational method
in Propositions \ref{prop:4-4} and \ref{prop:4-6}.
Section \ref{sec:five} is our main section in which we extend
the theory in Reference \cite{LZZ}  to that for a multi-phase flow
and obtain the Euler equation with the surface tension
of the multi-phase field in Theorem \ref{th:5-2}.
The extended Euler equation for the multi-phase flow
is derived from the variational
principle of the action integral in Theorem  \ref{th:5-1}.
As a special case, we also derive the 
Euler equation to a two-phase field with wall effects
in Theorem \ref{th:5-3}.
In Section 6, 
using these methods in the computational fluid dynamics \cite{Ch,H,HN},
we consider numerical computations 
of the contact angle problem of a two-phase field
because the contact angle problem 
for the two-phase field circumscribed in a wall
is the simplest non-trivial triple junction problem.
By means of our scheme, for given surface tension coefficients, 
we show two examples of the numerical computations in which
the contact angles automatically appeared
without any geometrical constraints
 and any difficulties for the singularities at triple junctions.
The computations were very stable.
Precisely speaking, as far as we computed, 
the computations did not collapse for any boundary conditions
and for any initial conditions.



\section{Mathematical Preliminaries} \label{sec:two}

\subsection{Preliminary of surface theory} \label{sec:two-one}

In this subsection, we review the theory of surfaces 
from the viewpoint of low-dimensional differential geometry.
The interface problems have been also studied for last three decades
in pure mathematics,
which are considered as a revision of the 
classical differential geometry \cite{E} from a modern point of view
\cite{ES,FW,GMO,M,T},
{\it{ e.g.}}, generalizations of the Weierstrass-Ennpper
theory of the minimal surfaces, isothermal surfaces,
constant curvature surfaces, constant mean curvature surfaces,
Willmore surfaces and so on.  
They are also closely connected with
the harmonic map theory 
and the theory of the variational principle \cite{FW,GMO}.


We consider a smooth surface $S$ embedded in 
three dimensional Euclidean space $\EE^3$.
Let $x = (x^1,x^2,x^3)$ be of the Cartesian coordinate system
and represent a point in $\EE^3$, and 
let the surface $S$ be locally expressed by a local parameter $(s^1,s^2)$.
We assume that the surface $S$ is expressed by
zeros of a real valued smooth function $q$ over $\EE^3$, {\it{i.e.}},
$$
	q(x)=0,
$$
such that in the region whose $|q|$ is sufficiently small
($|q| < \epsT$ for a positive number $\epsT>0$),
$|dq|$ agrees with the infinitesimal length in the Euclidean space.
Then $dq$ means the normal co-vector field (one-form), {\it{ i.e.}},
for the tangent vector field
$e_\alpha:=\partial_\alpha:=\partial/\partial s^\alpha$
($\alpha = 1, 2$) of $S$,
\begin{equation}
	\langle\partial_\alpha, dq\rangle=0 \quad \mbox{ over }\quad
               S = \{ x \in \EE^3 \ | q(x) = 0\}.
\label{eq:ddq}
\end{equation} 
Here $\langle, \rangle$ means the pointwise pairing 
between the cotangent bundle and the tangent bundle of $\EE^3$.
The function $q$
 can be locally regarded as so-called the level-set function 
\cite{S,ZMOW}.
We could redefine the domain of $q$ such that it is restricted to
a tubular neighborhood $T_S$ of $S$,
$$
      T_S:=\{x \in \EE^3 \ | \ |q(x)|<\epsT \}.
$$ 
Over $T_S$, $q$
agrees with the level-set function of $S$.
There we can naturally define a projection map
$\pi: T_S \to S$ and then we can regard $T_S$ as a fiber bundle over
$S$, which is homeomorphic to the normal bundle $N_{S}\to S$.
However the level-set function is defined as a signed 
distance function which is a global function over $\EE^3$
as a continuous function \cite{S} and thus it has no natural
projective structure in general;
for example,
the level-set function $L$ of a sphere with radius $a$ is given by
$$
L(x^1,x^2,x^3) = \sqrt{(x^1)^2 +(x^2)^2 + (x^3)^2} - a,
$$ 
which induces the natural projective (fiber) structure but 
the origin $(0, 0, 0)$ in the sphere case.  
The level-set function has no projective structure 
at $(0, 0, 0)$ in this case, and
we can not define its differential there.
In other words,
the level-set function is not a global function over $\EE^3$ 
as a smooth function in general.

  When we use the strategy of the 
fiber bundle and its connection, we restrict ourselves to
consider the function $q$ in  $T_S$. 
Then the relation (\ref{eq:ddq}) and the parameter
$(s_1, s_2)$ are naturally lifted 
to $T_S$ as an inverse image of $\pi$.

Further for $e_q:=\partial_q:=\partial/\partial q$, we have
$$
	\partial_\alpha (e_q) = \sum_\beta \Gamma^\beta_{\ \alpha q} e_\beta
     \mbox{ over } S.
$$
Here 
$(\Gamma^\beta_{\ \alpha q})$ is the Weingarten map, which is
a kind of a point-wise $2\times2$-matrix 
$((\Gamma^\beta_{\alpha q})_{\alpha\beta})$ \cite[Chapter VII]{KN}.
The eigenvalue of 
$(\Gamma^\beta_{\alpha q})$ is the principal
curvature, whereas a half of its trace
$\mbox{tr}(\Gamma^\beta_{\alpha q})/2$ is known as the 
mean curvature and its determinant 
$\mbox{det}(\Gamma^\beta_{\alpha q})$
means the Gauss curvature \cite[Chapter VII]{KN}.

Noting  the relation,
$\langle e_\beta, d s^\alpha\rangle = \delta_\beta^\alpha$ 
for $\alpha, \beta = 1, 2$, 
 the twice of the mean curvature, $\kappa$, is given by,
$$
	\sum_{\alpha}  \partial_\alpha (e_q) d s^\alpha = \kappa
        \quad \mbox{over}\quad S.
$$
Further noting the relation $\partial_q e_q d q =0$, we obtain
$$
	 \sum_{\alpha}  \partial_\alpha (e_q) d s^\alpha 
	 + \partial_q(e_q) d q = \kappa
        \quad \mbox{over}\quad S.
$$
Due to  the flatness of the Euclidean space, we identify $e_q$ with 
$\nabla q /|\nabla q|$ and then we have  the following proposition.

\begin{proposition} \label{prop:2-1}
The following relation holds at a point over $S$,
$$
	\mathrm{div}\left( \frac{\nabla q }{|\nabla q|}\right) = \kappa.
$$
\end{proposition}
For the case $|\nabla q| = 1$,
using the Hodge star operator \cite{AM,N} and the exterior
derivative $d$, we also have 
an alternative expression $*d * dq =\kappa$ over the surface $S$.
Here the Hodge star operator is $*: \Lambda^p(T_S) \to \Lambda^{3-p}(T_S)$
and the exterior derivative $d: \Lambda^p(T_S) \to \Lambda^{p+1}(T_S)$
$(d \omega = \sum_{i=1}^3\partial_i \omega d x^i)$,
where $\Lambda^p(T_S)$ is the set of smooth $p$-forms over
$T_S$ \cite{N}.

Noting that as the left hand side of formula in Proposition \ref{prop:2-1}
can be lifted to $T_S$,
the formula plays an important role in References \cite{BKZ,LZZ,ZCMO}
and in this article.

\subsection{Preliminary of Noether's theorem} \label{sec:two-two}

In this subsection,
 we review Noether's theorem in the variational method which
appears in a computation of the energy-momentum tensor-field
in the classical field theory \cite{AM,BGG,IZ}.

Let  the set of $\ell$ smooth real-valued functions
 over $n$-dimensional Euclidean space $\EE^n$
 be denoted by $\cC^{\infty}(\EE^n)^{\otimes\ell}$,
where $n$ is mainly three.
Let $x = (x^1,x^2,\ldots,x^n)$ be of the Cartesian coordinate system
of $\EE^n$.
We consider the functional 
$I: \cC^{\infty}(\EE^n)^{\otimes\ell} \to \RR$,
\begin{equation}
	I = \int_{\EE^n} d^n x{ \cF}(\phi_a(x),\partial_i \phi_a(x)),
\label{eq:IinEL}
\end{equation}
where
$\cF$ is a local functional,
$\cF: \cC^{\infty}(\EE^n)^{\otimes\ell}|_x \to \Lambda^n(\EE^n)|_x$,
\begin{equation*}
\begin{split}
\cF: (\phi_a)_{a=1,\ldots,\ell}|_x \mapsto &
{\cF}(\phi_a(x),\partial_i \phi_a(x))d^n x
\equiv{ \cF}(\phi_a(x),\partial_1 \phi_a(x),
\ldots, \partial_n \phi_a(x))d^n x \\
&
\equiv{ \cF}(\phi_1(x), \ldots, \phi_\ell(x), 
\partial_1 \phi_1(x), \ldots, \partial_n \phi_\ell(x))d^n x \\
\end{split}
\end{equation*}
and $\partial_i := \partial/ \partial x^i$, $(i=1, \cdots, n)$.
Then we obviously have the the following
proposition.

\begin{proposition}\label{prop:A-1} 
For the functional $I$ in (\ref{eq:IinEL}) over
$\cC^{\infty}(\EE^n)^{\otimes\ell}$,
the Euler-Lagrange equation coming from the variation
with respect to $\phi_a$ of $(\phi_b)_{b=1, \ldots, \ell}
 \in \cC^{\infty}(\EE^n)^{\otimes\ell}$,
{\it{i.e.}}, 
$
	\frac{\delta I}{\delta \phi_a(x)} = 0
$, 
 is given by
\begin{equation}
	\frac{\delta \cF}{\delta \phi_a(x)}
	-\sum_{i=1}^n
\partial_i \frac{\delta \cF}{\delta \partial_i \phi_a(x)} = 0.
\label{eq:AppeA1}
\end{equation}
\end{proposition}

Using the equation (\ref{eq:AppeA1}), 
we consider an effect of a small translation
  $x$ to $x'= x + \delta x$ on the functional $I$.
The following proposition is known as Noether's theorem
which plays crucial roles in this article.

\begin{proposition}\label{prop:A-2} 
The functional derivative $I$ with respect to $\delta x_i$ is
given by
\begin{equation}
\frac{\delta I}{\delta x^i}
   = \sum_{j=1}^n\partial_j\left[ \sum_{a=1}^\ell
   \frac{\delta \cF}{\delta\partial_j \phi_a}\partial_i \phi_a \right]
         - \partial_i  \left[  { \cF}\right].
\label{eq:AppeA2}
\end{equation}
If $I$ is invariant for the translation, 
(\ref{eq:AppeA2}) gives the conservation of the momentum.
\end{proposition}

\begin{proof}
For the variation $x'= x + \delta x$, the scalar function becomes
\begin{equation*}
	\phi_a(x') = \phi_a(x) +\sum_{i=1}^n  \partial_i \phi_a(x)
	\delta x^i + O(\delta x^2).
\end{equation*}
From the relations on the Jacobian and each component,
\begin{equation*}
	\frac{\partial x'}{\partial x} 
          = 1 +\sum_{i=1}^n \partial_i \delta x^i + O(\delta x^2),
\quad
	\frac{\partial x^k}{\partial x'{}^i}=\delta^k_i
	 -\partial_i \delta x^k + O(\delta x^2),
\end{equation*}
we have
\begin{equation*}
\begin{split}
	\frac{\partial \phi_a(x')}{\partial x'{}^i}
      &=\frac{\partial \phi_a(x) +
         \sum_{j=1}^n \partial_j \phi_a(x) \delta x^j}{\partial x^k}
       \frac{\partial x^k}{\partial x'{}^i} + O(\delta x^2)\\
      &=\partial_i \phi_a+
       \sum_{j=1}^n (\partial_i\partial_j \phi_a) \delta x^j  + O(\delta x^2).
\end{split}
\end{equation*}
Then up to $\delta x^2$, we obtain 
\begin{equation*}
\begin{split}
&  \int_{\EE^n} d^n x' {\cF}(\phi_a(x'),\partial_i' \phi_a(x'))
-  \int_{\EE^n} d^n x {\cF}(\phi_a(x),\partial_i \phi_a(x))\\
&= \int_{\EE^n}\Bigr[
\sum_{i=1}^n\sum_{a=1}^\ell 
\frac{\delta {\cF}}{\delta \phi_a} \partial_i \phi_a(x) 
\delta x^i +\sum_{i,j = 1}^n \sum_{a=1}^\ell
\frac{\delta \cF}{\delta\partial_j \phi_a} 
\partial_i \partial_j\phi_a(x) \delta x^i
+\sum_{j=1}^n{\cF} \partial_i \delta x^i\Bigr] d^n x\\
 &=\int_{\EE^n}\left(
\sum_{i=1}^n \partial_i\left[ \sum_{j=1}^n \sum_{a=1}^\ell
\frac{\delta \cF}{\delta\partial_i \phi_a}\partial_i \phi_a 
            -{ \cF} \right] \delta x^i \right) d^n x.\\
            \end{split}
\end{equation*}
Here we use the Euler-Lagrange equation (\ref{eq:AppeA1}) and then
we have  (\ref{eq:AppeA2}).
If we assume that $I$ is invariant for the variation, it vanishes.
\qed
\end{proof}

\section{Variational principle for incompressible fluid dynamics}
\label{sec:three}

As we will derive the governing equation as the variational
equation of an incompressible multi-phase flow with interfaces
using the variational method,
let us review the variational theory of the
incompressible fluid to obtain the Euler equation
following References \cite{Ar,AK,EM,K,Ko,MW,NHK}.

Let $\Omega$ be a smooth domain in $\EE^3$.
The incompressible fluid dynamics can be interpreted
as a geometrical problem associated with an infinite dimensional Lie group
\cite{AK,EM,O}.
It is related to the volume-preserving diffeomorphism  
group $\SDiff(\Omega)$
as a subgroup of  the diffeomorphism group
$\Diff(\Omega)$. The diffeomorphism group $\Diff(\Omega)$
is generated by a smooth coordinate transformation of $\Omega$.
The Lie algebras $\sdiff(\Omega) \equiv T_e\SDiff(\Omega)$ of $\SDiff(\Omega)$
and $\diff(\Omega) \equiv T_e\Diff(\Omega)$ of $\Diff(\Omega)$ are
 the infinite dimensional real vector spaces.
 The $\sdiff(\Omega)$  is a linear subspace of $\diff(\Omega)$.

Following Ebin and Marsden \cite{EM}, we consider 
the geometrical meaning of the action integral of
an incompressible fluid,
\begin{equation}
	\int_T dt \int_{\Omega} d^3 x 
\left(\frac{1}{2}\rho |u|^2\right).
\label{eq:Eue}
\end{equation}
Here 
$T:=(0, T_0)$ is a subset of the set of real numbers $\RR$,
$(x, t)$ is the Cartesian coordinate of the space-time 
$\Omega\times T$, $\rho$ is the density of the fluid
which is constant in this section, and $u=(u^1, u^2, u^3)$ is
the velocity field of the fluid.

Geometrically speaking,
a flow obeying the incompressible fluid dynamics is
considered  as a section of
a principal bundle
$\IFluid(\Omega\times T)$ over the absolute time axis $T \subset \RR$ 
as its base space,
\begin{equation}
\begin{CD}
	\SDiff(\Omega) @>>> \IFluid(\Omega\times T) \\
         @.                   @V{\varpi}VV\\
        @.                     T.\\
\end{CD}
\label{eq:PBIFluid}
\end{equation}
The projection $\varpi$ is induced from the
trivial fiber structure $\varpi_\Omega: \Omega \times T 
\to T$, $((x, t) \to t)$. In the classical (non-relativistic) mechanics,
every point of space-time has a unique absolute time $t\in \RR$,
which is contrast to one in the relativistic theory.

Due to the Weierstrass polynomial approximation theorem \cite{Y},
 we can locally approximate
a smooth function by a regular function.
Let  the set of smooth functions
over $\Omega$ be denoted by $\cC^\infty(\Omega)$
and the set of the regular real functions 
by $\cC^\omega(\Omega)$
whose element can be expressed by the Taylor expansion
in terms of local coordinates.

The action of $\Diff(\Omega)$ on 
$\cC^\omega(\Omega) \subset\cC^\infty(\Omega)$
is given by
$$
	\ee^{s u^i \partial_i} f(x) = f(x + s u),
$$
for an element $f \in \cC^\omega(\Omega)$, and small $s>0$,
where $\partial_i := \partial / \partial x^i$ and 
we use the Einstein convention; when an index $i$ appears twice,
we sum over the index $i$.
Thus the action $\ee^{s u^i \partial_i}$
is regarded as an element of $\Diff(\Omega)$.

As a frame bundle of the principal bundle $\IFluid(\Omega\times T)$, we
consider a vector bundle $\Coor(\Omega\times T)$ with infinite rank,
$$
\begin{CD}
	\cC^\infty(\Omega) @>>> \Coor(\Omega\times T) \\
         @.                   @V{\varpi'}VV\\
        @.                     T.\\
\end{CD}
$$
Since $\cC^\infty(\Omega)$ is regarded as a non-countably infinite
dimensional linear space over $\RR$,
we should regard  
$\Diff(\Omega)$ and $\SDiff(\Omega)$  as subgroups of
an infinite dimensional general linear group if defined.

More rigorously, we should
consider the ILH space (inverse limit of Hilbert space) 
(or ILB space (inverse limit of Banach space)) introduced in 
Reference \cite{O} by adding a certain
topology to (a subspace of) $\cC^\infty(\Omega\times T)$,
and then we also should regard $\Diff$ and $\SDiff$  
as an ILH Lie group.
However our purpose is to obtain an extended Euler equation from a more
practical viewpoint.
Thus we formulate the theory primitively even though
we give up to consider a general solution for a general initial
condition.

We consider smooth sections of $\Coor(\Omega\times T)$ and
$\IFluid(\Omega\times T)$.
Smooth sections of $\Coor(\Omega\times T)$ can be
realized as $\cC^\infty(\Omega \times T)$.
In the meaning of the Weierstrass polynomial approximation theorem \cite{Y},
an appropriate topology in $\cC^\infty(\Omega\times T)$ makes
$\cC^\omega(\Omega\times T)$ dense in  $\cC^\infty(\Omega\times T)$
by restricting the region $\Omega \times T$ appropriately.
Under the assumption, we also deal with a smooth section of
$\IFluid(\Omega\times T)$.

Let us consider a coordinate function 
$(\gamma^i(x, t))_{i=1,2,3} \in \cC^\omega(\Omega \times T)$
such that
$$
	\frac{d}{d t} \gamma^i(x, t) = u^i(x, t), \quad
	\gamma^i(x, t) = x^i \ \ \mbox{at } \ t \in T,
$$
which means 
$$
\gamma^i(x, t + \delta t) = x^i + u^{i}(x,t) \delta t + O(\delta t^2),
$$
for a small $\delta t$.
Here the addition is given as a Euclidean move in $\EE^3$.
As an inverse function of $\gamma=\gamma(u,t)$, 
we could regard $u$ as a function of $\gamma$ and $t$,
$$
u(x,t) =u(\gamma(x,t),t).
$$
Further we introduce a small quantity modeled on $\delta t\cdot u^i$,
\begin{equation}
	\tgamma^i(x,t) := \gamma^i(x,t) - x^i.
\label{eq:tgamma}
\end{equation}
Then a section $g$
of $\IFluid(\Omega\times T)$ at $t \in T$ can written by,
\begin{equation}
	g(t) = \ee^{\tgamma^i \partial_i} \in \IFluid(\Omega\times T)\Bigr|_{t}
        \approx \SDiff(\Omega) \subset \Diff(\Omega).
\label{eq:etgamma}
\end{equation}
Here we consider $g$ as an element of $\SDiff(\Omega)$ and thus
it satisfies the condition of
the volume preserving, which appears as the constraint that the Jacobian,
$$
\frac{\partial \gamma}{\partial x} 
:=
\det\left(\frac{\partial \gamma^i}{\partial x^j} \right)
= (1 + 
\tr(\partial_j u^i) \delta t) + O(\delta t^2),
$$
must preserve $1$, {\it{i.e.}}, the well-known condition
that $\tr (\partial_j u^i) = \mathrm{div}(u)$ must vanish,
or $\displaystyle{\frac{d}{dt}\frac{\partial \gamma}{\partial x} = 0}$.

Following Reference \cite{EM}, we reformulate the
action integral (\ref{eq:Eue})
as  {\lq\lq}the energy functional{\rq\rq}
in the frame work of the harmonic
map theory.
In the harmonic map theory \cite{GMO}
by considering a smooth map $h : M \to G$
for a $n$-smooth base manifold $M$ and its target group manifold $G$,
{\lq\lq}the energy functional{\rq\rq} is given by
\begin{equation}
    E = \frac{1}{2}\int_M \tr \left((h^{-1} d h )*(h^{-1} d h )\right).
    \label{eq:B-1}
\end{equation}
Here $*$ means the Hodge star operator, which is
for $*:TG\otimes\Lambda^p(M) \to TG\otimes\Lambda^{n-p}(M)$ where
$\Lambda^p(M)$ is the set of the smooth $p$-forms over $M$ \cite{N},
and 
$TG\otimes\Lambda^p(M)$ is the set of the tangent bundle $TG$ valued
smooth $p$-forms over $M$ \cite{N}.
The term {\lq\lq}energy functional{\rq\rq} in the harmonic map theory
means that it is an invariance of the system and thus it sometimes
differs from an actual energy in physics.

Since in (\ref{eq:PBIFluid}), the base space $T$ is one-dimensional and
the target space $\IFluid(\Omega\times T)|_{t}$ at $t \in T$ is 
the infinite dimensional space,
{\lq\lq}the energy functional{\rq\rq} (\ref{eq:B-1})
in the harmonic map theory corresponds to
the action integral $\cS_{\free} [\gamma]$ which is defined by
$$ 
\cS_{\free} [\gamma]=
\frac{1}{2} \int _T \int_\Omega 
\frac{\partial \gamma}{\partial x} 
\rho d^3 x \cdot dx^i \otimes dx^i \left(
	\left( \ee^{-\tgamma^k \partial_k}
       dt \frac{d}{dt}\ee^{\tgamma^\ell \partial_\ell} \right)
	\left( \ee^{-\tgamma^j \partial_j}
        \frac{d}{dt}\ee^{\tgamma^n \partial_n} \right)
         \right).
$$
Here $d x^i (\partial_j ) := 
\langle\partial_j ,  d x^i \rangle
 = \delta^i_{\ j}$ is the natural pairing  between
 $T \Omega$ and $T^* \Omega$.
The trace in  (\ref{eq:B-1}) corresponds to the 
integral over $\Omega$ with 
$ \frac{\partial \gamma}{\partial x} 
\rho d^3 x \cdot dx^i \otimes dx^i$.
The Hodge $*$ operator acts on the element such as
$*\left( \ee^{-\tgamma^k \partial_k}
 dt \frac{d}{dt}\ee^{\tgamma^\ell \partial_\ell} \right)
= \left( \ee^{-\tgamma^k \partial_k}
        \frac{d}{dt}\ee^{\tgamma^\ell \partial_\ell} \right)$
as the natural map from
$\diff(\Omega)$ valued 1-form to 0-form.
Further we assume that $\rho$ is a constant function in this section.
Then the  action integral $\cS_{\free} [\gamma]$
obviously agrees with (\ref{eq:Eue}).

We investigate the functional derivative and the
variational principle of this $\cS_\free[\gamma]$.
Let us consider the variation,
$$
	\gamma^j(x,t') = \gamma^j(x, t') + \delta \gamma^j(x,t'), \quad
\mbox{and} \quad
	\tgamma^j(x,t') = \tgamma^j(x, t') + \delta \gamma^j(x,t'), \quad
$$
where we implicitly assume that $\delta \gamma^j$ is proportional
to the Dirac $\delta$ function,
$\delta(t' - t)$, for some $t$ and $\delta \gamma^j$ vanishes at
$\partial \Omega$.
As we have concerns only for local effects or differential
equations, we implicitly assume that
we can neglect the boundary effect arising from  $\partial \Omega$ on
the variational equation.
If one needs the boundary effect, he would follow the study of
Shkoller \cite{Shk}.
Further one could use the language of the sheaf theory to
describe the local effects \cite{KKK}.
As we are concerned only with differential equation and thus
our theory is completely local except Section \ref{sec:VOF}, we could
deal with germs of related bundles \cite{AGLV} as in Reference \cite{M},
which is also naturally connected with
 a computational method of fluid dynamics \cite{M2}.

Let us consider the extremal point of the action integral (\ref{eq:Eue})
following the variational principle.
Noting that $\partial \gamma /\partial x=1$, 
the above Jacobian becomes
$$
\frac{\partial (\gamma + \delta \gamma)}{\partial x} 
=
\frac{\partial \gamma}{\partial x} 
(1 + \partial_k \delta \gamma^k) + O((\delta \gamma)^2).
$$
Since we  employ the projection method, we firstly
consider a variation
in $\diff(\Omega)$ rather than $\sdiff(\Omega)$.
For the variation,
the  action integral $\cS_{\free} [\gamma]$ with (\ref{eq:etgamma})
becomes 
\begin{equation}
\begin{split}
\cS_\free [\gamma+\delta  \gamma] -&
\cS_\free [\gamma] =\\
    &-\int_T dt \int_\Omega 
\frac{\partial \gamma}{\partial x} 
d^3 x \cdot dx^i \otimes dx^i 
 \left(
	\delta \gamma^k
        \frac{d}{dt}\left( \rho
         g^{-1} \frac{d}{dt} g \right)
+ \delta \gamma^k \partial_k \frac{1}{2}\rho |u|^2 
 \right).\\
\end{split}
\nonumber
\end{equation}

Now we have the following proposition.

\begin{proposition} \label{prop:3-1}
Using the above definitions, the variational principle
in $\SDiff(\Omega)$,
$$
\frac{\delta \cS_\free[\gamma]}{\delta \gamma(x,t)}
   \Bigr|_{\SDiff(\Omega)|_t} = 0,
$$
is reduced to the Euler equation,
\begin{equation}
 \frac{\partial}{\partial t} \rho u^i + u^j \partial_j \rho u^i
          + \partial_i p = 0,
\label{eq:EulerEq}
\end{equation}
where $p$ comes from the projection from $T\Diff(\Omega)|_{\SDiff(\Omega)}
\to T\SDiff(\Omega)$.
\end{proposition}

\begin{proof}
Basically we leave the rigorous proof
and especially the derivation of $p$ to \cite{AK,EM}.
The existence of $p$ was investigated well in 
Appendix of Reference \cite{EM}
as the Hodge decomposition \cite{AM,N}. (See also
 the following Remark \ref{rmk:3-2}.)
Except the derivation of $p$,
we use the above relations and the following relations,
\begin{equation*}
\begin{split}
        \frac{d}{dt}
	\left(\rho  \ee^{-\tgamma^j \partial_j}
        \frac{d}{dt}\ee^{\tgamma^n \partial_n} \right) 
    &=
	\frac{d}{dt} \left(\rho u^i(\gamma(t),t) \partial_i \right) \\
   &= \left( \frac{\partial}{\partial t}\rho u^i|_{x=\gamma}
       + (\frac{d}{dt}\tgamma^j) \partial_j \rho u^i \right) \partial_i \\
   &= \left( \frac{\partial}{\partial t}\rho u^i
       + u^j \partial_j\rho u^i \right) \partial_i \\
   &=:\left(\frac{D}{Dt} \rho u^i \right) \partial_i. \\
\end{split}
\end{equation*}
Then we obtain the Euler equation.
\qed
\end{proof}

\begin{myremark}\label{rmk:3-2}
 {\rm{
The Euler equation was obtained by the simple variational principle.
Physically speaking, the conservation of 
the momentum in the sense of Noether's theorem \cite{BGG,IZ} led to
the Euler equation.
However, we could introduce the pressure $p_L$ term as the
Lagrange multiplier of the constraint of the volume preserving.
In the case, instead of $\cS_\free$, we deal with 
$$
\cS_{\free,p} = \cS_\free
 + \int_T dt \int_\Omega p_L(x,t) \frac{\partial \gamma}{\partial x} 
                    d^3 x.
$$
Then  noting the term coming from the Jacobian, the relation,
$$
	\frac{\delta \cS_{\free,p}
             [\gamma]}{\delta\gamma(x,t)}\Bigr|_{\SDiff(\Omega)|_t} = 0,
$$
is reduced to the Euler equation,
$$
 \frac{\partial}{\partial t} \rho u^i + u^j \partial_j \rho u^i
          + \partial_i (p_L + \frac{1}{2}\rho |u|^2) = 0.
$$
As the pressure is determined by the (divergence free) condition of $u$,
we  renormalize \cite[(25)]{Ko},
$$
	p := p_L + \frac{1}{2}\rho |u|^2.
$$
More rigorous arguments are left to  References \cite{EM,O} 
and physically interpretations are, {\it {e.g.}}, in References
\cite{AF,B,K,NHK,Schm,V}.

We give a comment on 
the projection from $T\Diff(\Omega)|_{\SDiff(\Omega)}
\to T\SDiff(\Omega)$ in (\ref{eq:EulerEq}),
 which is known as the projection method.
First we note that
the divergence free condition $\mathrm{div} (u) =0$ 
simplifies the Euler equation (\ref{eq:EulerEq}),
$$
\rho \frac{D u}{D t} + \nabla p = 0, \quad 
\frac{\partial u^i}{\partial t} + u^j \partial_j u^i
+\frac{1}{\rho}\partial_i p=0.
$$
As mention in Section \ref{sec:VOF}, in the difference equation 
we have a natural interpretation of the projection method \cite{Cho}.
We, thus, regard $D u / Dt$ in $T\Diff(\Omega)|_{\SDiff(\Omega)}$ as 
$\lim_{\delta t \to 0}\frac{u(t + \delta t)-u(t)}{\delta t}$
for $u(t+\delta t):=u(t+\delta t,\gamma(t+\delta t)) \in \diff(\Omega)$ 
and $u(t):=u(t,\gamma(t)) \in \sdiff(\Omega)$, {\it{i.e.}},
$\mathrm{div} \left(u(t)\right) = 0$ by considering
$T\Diff(\Omega)$ at the unit $e$ of $\Diff(\Omega)$
 up to $\delta x^2$, as we did in (\ref{eq:tgamma}) and (\ref{eq:etgamma}).
In order to find the deformation 
$u^\parallel(t + \delta t)$ in $\sdiff(\Omega)$ by a natural projection
from $\diff(\Omega)$ to $\sdiff(\Omega)$ \cite[,p.36]{CM},
we decompose $u(t + \delta t)$ into $u^\parallel(t + \delta t)$ and
$u^\perp(t + \delta t)$ such that 
$\partial_i u^{\perp i}(t + \delta t) := \partial_i u^i(t + \delta t)$.
Then $u^\parallel(t + \delta t):= u(t + \delta t)-$
$u^\perp(t + \delta t)$ belongs to $\sdiff(\Omega)$.
Thus the pressure $p$ is determined by \cite{CM}
\begin{equation}
\partial_i u^i(t + \delta t)
+ \delta t\partial_i \frac{1}{\rho}\partial_i p =0.\quad
\label{eq:ProjM}
\end{equation}
In other words, since $u^\parallel(t + \delta t)\equiv$
$u^i(t + \delta t) + \delta t \frac{1}{\rho}\partial_i p$ belongs
to $\sdiff(\Omega)$,
the deformation of
$u^{\parallel i}(t + \delta t)-  u^i(t)$ which gives
$D u^\parallel/D t$ and the Euler equation (\ref{eq:EulerEq}) 
is the deformation in $\IFluid(\Omega \times T)$.
After taking the continuous limit $\delta t \to 0$,
the equation for the pressure
(\ref{eq:ProjM}) can be written as \cite{Cho},
$$
(\partial_i u^j) (\partial_j u^i)
+\partial_i\frac{1}{\rho}\partial_i p=0,
$$
by noting the relations $[\partial_t, \partial_i]=0$
and $\mathrm{div}(u(t))=0$, {\it{i.e.}},
$\partial_i u^i(t + \delta t) = \partial_i[ u^i(t) $
$+ \frac{\partial}{\partial t} u^i(t) \delta t$
$+ u^j(t)\partial_j u^i(t)\delta t] + O(\delta t^2)$.
The Poisson equation with (\ref{eq:EulerEq})
guarantees the divergence free condition.
Hence the pressure $p$ in the incompressible fluid
is determined geometrically.
}}
\end{myremark}

\section{Reformulation of Surface tension as a minimal surface energy}
\label{sec:four}

In this section we  reformulate  
the SURFER scheme \cite{LZZ} following the 
variational principle and the arguments of previous sections.

\subsection{Analytic expression of surface area}\label{subsec:fourOne}

We first should note that in general, the higher dimensional 
generalized function like the Dirac delta function has some
difficulties in its  definition \cite{Y}. 
For the difficulties,
in the Sato hyperfunctions theory \cite{KKK},
the sheaf theory and the cohomology theory are
necessary to the descriptions of the higher dimensional generalized functions,
which are too abstract to be applied to a problem with an arbitrary
geometrical setting.
Even for the generalized function in the framework of
Schwartz distribution theory, we should pay attentions on its treatment.
However since the surface $S$ in this article is 
a hypersurface and its codimension is one, 
the situation makes the problems much easier.

We assume that the smooth surface $S$ is orientable and compact such
 that we could
define its inner side and outer side. In other words, there
is a three dimensional subspace (a manifold with boundary) $B$ such that
its boundary $\partial B$ agrees with $S$ and $B$ 
is equal to the inner side of $S$ with $S$ itself.
Then we consider a generalized function $\theta$ over 
$\Omega \subset \EE^3$ such that
it vanishes over the complement $B^c = \Omega \setminus B$
 and is unity for the interior $B^\circ:=
B \setminus \partial B$;
$\theta$ is known as a characteristic function of $B$.

We consider the global function 
$\theta(x)$ and its derivative $d\theta(x)$ in the sense of
the generalized function, which is given by
\begin{equation*}
d \theta(x) = \sum_{i}\partial_i \theta(x) d x^i =\partial_q \theta(x) d q.
\end{equation*}
Here we use the notations in Section \ref{sec:two-one}.
Using the nabla symbol $\nabla \theta=(\partial_i \theta(x) )_{i=1,2,3}$,
$|\nabla \theta| d^3 x$ is interpreted as
\begin{equation*}
	|\nabla \theta|d^3 x=|(* d \theta)\wedge dq|. \quad
\end{equation*}
Here due to the Hodge star operation 
$*: \Lambda^p(\Omega) \to \Lambda^{3-p}(\Omega)$,
$* d \theta = \tilde e \partial_q \theta d s^1\wedge d s^2$
where $\tilde e$ is the Jacobian between the coordinate systems
 $(d s^1, d s^2, d q)$ and $(d x^1, d x^2, d x^3)$.
Then we have the following proposition;
\begin{proposition}
If the integral,
\begin{equation*}
\cA:=	\int_{\Omega} |\nabla \theta| d^3 x \equiv \int_{\Omega} 
       |(* d \theta)\wedge dq|,
\end{equation*}
is finite, $\cA$ agrees with the area of the surface $S$.
\end{proposition}

It should be noted that due to the codimension of 
$S \subset \Omega$,
we have used the fact that the Dirac $\delta$ function 
along $q \in T_S$ is the integrable
 function whose integral is the Heaviside function.
This fact is a key of this approach.

\subsection{Quasi-characteristic function for surface area}

For the later convenience, we introduce a support of 
a function over $\Omega$, which 
is denoted by {\lq\lq}supp{\rq\rq}, {\it{i.e.}},
for a function $g$ over $\Omega$, its support is defined by
$$
	\supp(g) =\overline{\{x \in \Omega \ | \ g(x) \neq 0\}},
$$
where {\lq\lq}$\ \bar{\  } \ ${\rq\rq}
 means the closure as the topological space $\Omega$.

One of our purposes is to express the surface $S$ by means of 
numerical methods, approximately.
Since it is difficult  to deal with the generalized function $\theta$
in a discrete system
like the structure lattice \cite{Ch},
we introduce a smooth function $\xi$ over $\Omega$
as a quasi-characteristic function
which approximates the function
$\theta$
\cite{BKZ,LZZ},
\begin{equation}
     \xi(x) =\left\{ 
\begin{matrix} 
              0 & \mbox{for } x \in B^c \bigcap \{x \in \Omega \ | \ 
                                 |q(x)| < \epsX/2\}^c, \\
              1 & \mbox{for } x \in B \bigcap \{x  \in \Omega\ | \
                                 |q(x)| < \epsX/2\}^c, \\
              \mbox{monotonically increasing in } q(x)& \ \mbox{otherwise}. 
                          \end{matrix}
    \right.
\label{eq:epsX}
\end{equation}
We note that along the line of $d q$ for $q \in (-\epsX/2, \epsX/2)$,
$\xi$ is a monotonically increasing function which interpolates 
between $0$ and $1$.
We now implicitly assume that $\epsX$ is much smaller than
$\epsT$ defined in Section \ref{sec:two-one} so that
support of $|\nabla \xi|$ is in the tubular neighborhood $T_S$.
However after formulating
the theory, we extend the geometrical setting in Section \ref{sec:two-one}
to more general ones which include singularities; there
$\epsT$ might lose its mathematical meaning but $\epsX$
survives as a control parameter which governs the system.
For example, as in Reference \cite{LZZ}, 
we can also deal with a topology change well.

By letting $\xi^c (x) := 1 - \xi(x)$, $\xi^c$ and $\xi$ are
regarded as the partition of unity \cite[I p.272]{KN}, or
$$
	\xi(x) + \xi^c(x) \equiv 1.
$$
We call these $\xi$ and $\xi^c$  {\lq\lq}color functions{\rq\rq} or 
{\lq\lq}phase fields{\rq\rq} in the following sections.
We have an approximation
of the area of the surface $S$ by the following proposition.

\begin{proposition}
Depending upon $\epsX$, we define the integral,
\begin{equation*}
     \cA_{\xi}:=\int_{\Omega} |\nabla \xi| d^3 x, 
\end{equation*}
and then the following inequality holds,
\begin{equation*}
\quad |\cA_\xi -\cA | < \epsX \cdot \cA.
\end{equation*}
\end{proposition}

Here we note that
$\cA_\xi$ is regarded as the approximation of the area $\cA$ of $S$
controlled by $\epsX$. In other words,
we use $\epsX$ as the parameter which controls the
difference between the characteristic function $\theta$ and 
the quasi-characteristic function $\xi$
in the phase field model \cite{BKZ,LZZ}.

Let us consider its extremal point following the 
variational principle in a purely geometrical sense.

\begin{proposition} \label{prop:4-3}
For sufficiently small $\epsX$, we have
\begin{equation*}
\begin{split}
	\frac{\delta}{\delta \xi(x)} \cA_\xi 
         &=-\partial_i\frac{\partial_i \xi }{|\nabla \xi| }(x)
                        \\
       &=\kappa(x),
\end{split}
\end{equation*}
where $x \in S$ or $q=0$. 
\end{proposition}

\begin{proof}
Noting the facts that $\partial \xi/\partial q <0$ at $q=0$ and
\begin{equation*}
	|\nabla \xi| = \sqrt{\nabla \xi\cdot \nabla \xi},
\end{equation*}
Proposition \ref{prop:A-1} and
the equality in Proposition \ref{prop:2-1} show the relation.
\qed
\end{proof}

In the vicinity of $S$, $q$ in Section \ref{sec:two-one}
could
be identified with the level-set function and 
the authors in References \cite{ZCMO,ZMOW} also used this relation.
Since all of geometrical quantities on $S$ are
lifted to $T_S$ as the inverse image of $\pi$,
the relation in Proposition \ref{prop:4-3} is also
defined over $(\supp(|\nabla \xi|))^\circ \subset  T_S$
and we redefine the $\kappa$ by the relation from here.

\bigskip
\subsection{Statics}
Let us consider physical problems as we finish the geometrical
setting.
Before we consider dynamics of the phase field flow,
we consider a statical surface problem.
Let $\sigma$ be the surface tension coefficient between two fluids
corresponding to $\xi$ and $\xi^c$.
Now let us call $\xi$ and $\xi^c$ {\lq\lq}color functions{\rq\rq} or 
{\lq\lq}phase fields{\rq\rq}. 
More precisely, we say that a color function with individual
physical parameters is a phase field.
The surface energy $\cE:=\sigma \cA$ is, then, approximately given by
\begin{equation}
	\cE_\two := \sigma \cA_\xi = \sigma\int_{\Omega} |\nabla \xi| d^3 x.
\end{equation}
As a statical mechanical problem, we consider
the variational method of this system following Section \ref{sec:two-two}.

Since a statical surface phenomenon is caused by the difference
of the pressure of each material, we now consider a free energy functional
\cite{MM},
\begin{equation}
	\cF_\two :=\int_{\Omega} 
     \left(\sigma|\nabla \xi| - (p_1 \xi + p_2 \xi^c)\right) d^3 x,
  \label{eq:ELs0}
\end{equation}
where $p_a$ ($a = 1, 2$) is the proper pressure of each material.

\begin{proposition}\label{prop:4-4}
The variational problem with respect to $\xi$, $\delta \cF_\two/\delta \xi =0$,
reproduces the Laplace equation {\rm{\cite[Chap.7]{LL}}},
\begin{equation}
  (p_1 - p_2)  - \sigma \kappa(x) = 0, \quad x\in 
        (\supp(|\nabla \xi|))^\circ.
  \label{eq:ELs}
\end{equation}
\end{proposition}
\begin{proof}
As in Proposition \ref{prop:A-1}, direct computations give the relation.
\qed
\end{proof}

This proposition implies that the functional $\cF_\two$ is natural. The
solutions of (\ref{eq:ELs}) are given by the constant mean curvature
surfaces studied in References \cite{FW,GMO,T}. 

Furthermore we also have another static equation,
whose relation to the Laplace equation (\ref{eq:ELs})
is written in Remark 
\ref{rmk:4-2}.
\begin{proposition}\label{prop:4-5}
For every point $x \in \Omega$,
the variation principle,
$\delta \cF_\two/\delta x^i=0$, gives
\begin{equation}
  \sigma\left(  \sum_j \partial_i\frac{\partial_j \xi 
              \partial_j \xi }{|\nabla \xi| }
     -  \sum_j\partial_j\frac{\partial_j \xi \partial_i \xi }{|\nabla \xi| }
 \right)
  - (p_1 - p_2) \partial_i \xi=0,
\label{eq:Amini0}
\end{equation}
or
\begin{equation}
\partial_j \tau_{ij}(x) - (p_1 - p_2) \partial_i \xi(x)=0,
\label{eq:Amini}
\end{equation}
where
$$
\tau(x) := \sigma\left(I - \frac{\nabla \xi}{|\nabla \xi|}
\otimes
\frac{\nabla \xi}{|\nabla \xi|}\right)
|\nabla \xi|(x).
$$
\end{proposition}

\begin{proof}
We are, now, concerned with  the variation 
$x \to x + \delta x$ for every point $x \in \Omega$.
We apply Proposition \ref{prop:A-2} to this case, {\it{i.e.}},
\begin{equation*}
\begin{split}
\frac{\delta \cF_\two}{\delta x^i}
&= -\sigma \left[\partial_i |\nabla \xi|- \partial_j\left(
         \partial_i \xi(x) \cdot\frac{\delta}{\delta\partial_j \xi(x)} 
           |\nabla \xi|  \right) \right](x)
 + (p_1 - p_2) \partial_i \xi(x),
\end{split}
\end{equation*}
by using (\ref{eq:ELs}) as its Euler-Lagrange equation
(\ref{eq:AppeA1}). Further for $x\not\in (\supp(|\nabla \xi|))^\circ$,
its Euler-Lagrange equation (\ref{eq:AppeA1}) gives a trivial
relation, {\it{i.e.}}, {\lq\lq}$0=0${\rq\rq}.
Then we have (\ref{eq:Amini}).
\qed
\end{proof}

\begin{myremark}\label{rmk:4-0}
{\rm{
It is worthwhile noting that (\ref{eq:Amini0}) and (\ref{eq:Amini}) 
are defined over $\Omega$
rather than $(\supp(|\nabla \xi|))^\circ$ because due to the relation,
$$
  | \partial_i \xi | \le |\nabla \xi|,
$$
even at the point at which denominators in the first term in (\ref{eq:Amini0}) 
vanish,
the first term is well-defined and vanishes.

Hence (\ref{eq:Amini0}) and (\ref{eq:Amini}) 
could be regarded as an extension of the defined region
of (\ref{eq:ELs}) to $\Omega$ and thus  
 (\ref{eq:Amini0}) and (\ref{eq:Amini}) 
have the advantage over (\ref{eq:ELs}).
The extension makes the handling of the surface tension much easier.
}}
\end{myremark}

\begin{myremark}\label{rmk:4-1}
{\rm{
In the statical mechanics, there appears a force $\partial_i \tau_{ij}$,
which agrees with one in  
(33) and (34) in Reference \cite{LZZ}
and (2.11) in Reference \cite{Jac}. 
We should note that in Reference \cite{Jac},
it was also stated that 
 this term is derived from the momentum conservation 
however there was not its derivation in detail. 
The derivation of the above $\tau$ needs the 
Euler-Lagrange equation (\ref{eq:AppeA1}), which corresponds to
 the Laplace equation (\ref{eq:ELs}) in this case,
when we apply Proposition \ref{prop:A-2} to
this system, though these objects did not appear in Reference \cite{Jac}. 
}}
\end{myremark}

\begin{myremark}\label{rmk:4-2}
{\rm{
In this remark, we comment on the identity between 
  (\ref{eq:ELs}) and (\ref{eq:Amini}).
Comparing these, we have the identity,
$$
	\partial_i \tau_{ij} = \sigma \partial_j \xi \cdot \kappa,
$$
which is, of course, obtained from the primitive computations.
It implies that (\ref{eq:Amini}) can be derived from the 
Laplace equation (\ref{eq:ELs}) with this relation. 
However it is worthwhile noting that both come from the 
variational principle in this article.
 In fact, when we handle multiple junctions,
we need a generalization of the Laplace equations over there 
like (\ref{eq:ELM3}), which is not easily obtained by taking the primitive
approach. 
Further the derivations from the variational principle
show their geometrical meaning in the sense of References \cite{Ar,AM,BGG}.
}}
\end{myremark}
\bigskip
\subsection{Dynamics}

Now we investigate the dynamics of the two-phase field.
There are two different liquids which are expressed by
phase fields $\xi$ and $\xi^c$ respectively.
We assume that they obey the incompressible fluid
dynamics. 
As in the previous section,
we consider the action of the volume-preserving diffeomorphism  
group $\SDiff(\Omega)$
on the color functions $\xi$ and $\xi^c$.
We extend the domain of $\xi$ and $\xi^c$  to $\Omega \times T$
and they are smooth sections of $\Coor(\Omega\times T)$.
For the given $t$, we will regard
$\xi$ and $\xi^c$ as functions of $\gamma^i$ in the previous section,
{\it{i.e.},} $\xi=\xi(\gamma(x,t))$.
For example, the density of the fluid is expressed by the relation,
 $$
\rho = \rho_1 \xi^c + \rho_2 \xi
$$ 
for constant proper densities $\rho_1$ and $\rho_2$ of the individual
liquids.
The density $\rho$, now, differs from a constant function
over $\Omega\times T$ in general.

We consider the action integral $\cS_{\two}$ 
including the surface energy,
\begin{equation}
	\cS_{\two}[\gamma]
     = \int_T d t\int_{\Omega}\left( \frac{1}{2}\rho |u|^2 -
               \sigma|\nabla \xi| + (p_1\xi + p_2 \xi^c) \right) d^3 x.
\label{eq:AI2d}
\end{equation}
The ratio between $\rho$ and $\sigma$
determines the ratio between the contributions
of the kinematic part and the potential (or surface energy) part
in the dynamics of the fluid.
Since the integrand in 
(\ref{eq:AI2d}) contains no $\partial \xi/\partial t$ term,
we obtain the same terms in the variational calculations
from the second and the third term in (\ref{eq:AI2d})
as those in (\ref{eq:ELs}) and (\ref{eq:Amini}) in 
 the static case even if we regard $n$ as $4$
and $x^4$ as $t$ in Section \ref{sec:two-two}.
By applying Proposition \ref{prop:A-1} to this system, we have
the following proposition as the Euler-Lagrange equation for $\xi$.

\begin{lemma}
The function derivative of $\cS_{\two}$ with respect to $\xi$ gives
\begin{equation}
	\frac{1}{2}(\rho_1 - \rho_2) |u(x,t)|^2
       + (p_1 - p_2) - \sigma \kappa (x,t)= 0, 
   \quad x\in 
        (\supp(|\nabla \xi|))^\circ,
\label{eq:EL2d}
\end{equation}
up to the volume preserving condition.
\end{lemma}
This could be interpreted as a generalization of the 
Laplace equation (\ref{eq:ELs}) as in the following remark.

\begin{myremark}{\rm{
Here we give some comments on the
generalized Laplace equation (\ref{eq:EL2d})
up to the volume preserving condition.
This relation (\ref{eq:EL2d})
does not look invariant for 
Galileo's transformation, $u \to u + u_0$ for a constant velocity
$u_0$.
However for the simplest problem of  Galileo's boost,
{\it{i.e.}}, static state on a system with a constant velocity $u_0$,
the equation (\ref{eq:EL2d}) gives
\begin{equation}
	\frac{1}{2}(\rho_1 - \rho_2) |u_0|^2
       + (p_1 - p_2) - \sigma \kappa (x,t)= 0, 
   \quad x\in 
        (\supp(|\nabla \xi|))^\circ,
\end{equation}
which might differ from the  Laplace equation (\ref{eq:ELs}).
However for the boost, we should transform $p_a$ into
\begin{equation}
	\tilde p_a := p_a + \frac{1}{2} \rho_a |u_0|^2.
\label{eq:GalileoB}
\end{equation}
Then the above equation of $\tilde p_a$ agrees with the static one
 (\ref{eq:ELs}). In other words (\ref{eq:GalileoB}) makes
 our theory invariant for the Gaililio's transformation.

For a more general case, we should regard $p_a$ as a function over
$\Omega \times T$ rather than a constant number due to
the volume preserving condition.
These values are contained in the pressure as mentioned in 
(\ref{eq:p_PLTwo}). 
The statement {\lq\lq}up to the volume preserving condition{\rq\rq}
 has the meaning in this sense.
In fact, in the numerical computation,
these individual pressures $p_a$'s are  not so important
as we see in Remark \ref{rmk:4-11}.
Due to the constraint of the incompressible
(volume-preserving) condition, the pressure $p$ is determined
as mentioned in Remark \ref{rmk:3-2}.
There are no contradictions with the 
Galileo's transformation and $\SDiff(\Omega)$-action.
}}
\end{myremark}

We consider the infinitesimal action of
$\SDiff(\Omega)$ around its identity.
As did in Section \ref{sec:three},
we apply the variational method to this
system in order to obtain the Euler equation with the surface tension.

\begin{proposition}\label{prop:4-6}
For every $(x,t) \in \Omega \times T$,
the variational principle,
$\delta \cS_\two/\delta \gamma^i(x,t) = 0$, gives the equation of motion,
or the Euler equation with the surface tension,
\begin{equation}
\begin{split}
        \frac{D \rho u^i}{D t} +
            \sigma \left( \sum_j \partial_i\frac{\partial_j \xi 
              \partial_j \xi }{|\nabla \xi| }
     -  \sum_j  \partial_j\frac{\partial_j \xi \partial_i \xi }{|\nabla \xi| }
     \right)
          + \partial_i p = 0.
\label{eq:Eulxi}
\end{split}
\end{equation}
Here $p$ is also the pressure coming from the effect of the 
volume-preserving.
\end{proposition}

\begin{proof}
The measure $d^3 x$ is regarded as 
$\displaystyle{\frac{\partial \gamma}{\partial x} d^3 x}$ with
$\displaystyle{\frac{\partial \gamma}{\partial x} = 1}$.
Noting $\displaystyle{\frac{d}{dt}\frac{\partial \gamma}{\partial x} = 0}$,
the proof in Proposition \ref{prop:3-1} and Remark \ref{rmk:3-2}
provide  the kinematic part with pressure term and 
Proposition \ref{prop:4-5} gives the remainder.
In this proof, the total pressure $p$ is defined in Remark \ref{rmk:4-11}.
\qed
\end{proof}

\begin{myremark} \label{rmk:4-11}
{\rm{
More rigorous speaking, as we did in Remark \ref{rmk:3-2},
we also renormalize the pressure,
\begin{equation}
\begin{split}
	p &= p_L + \frac{1}{2}\rho |u|^2 + p_1 \xi + p_2 \xi^c \\
	&= p_L + \frac{1}{2}(\rho_1 - \rho_2) \xi |u|^2 
+ (p_1 - p_2) \xi + \frac{1}{2}\rho_2 |u|^2 + p_2. 
\end{split}
\label{eq:p_PLTwo}
\end{equation}
As in Section \ref{sec:two-two}, the third term  in 
(\ref{eq:Eulxi}) includes the effects from $p_a$'s via the
generalized Laplace equation (\ref{eq:EL2d}) as the Euler-Lagrange
equation  (\ref{eq:AppeA1}).
}}
\end{myremark}

\begin{myremark} \label{rmk:7} {\rm{
\begin{enumerate}
\item
The equation of motion (\ref{eq:Eulxi}) is 
the same as (24) in Reference \cite{LZZ} basically.
We emphasize that  it is 
reproduced by   the variational principle.
\item
As in Reference \cite{LZZ}, in our framework, we can deal with the 
topology
changes and the singularities which are controlled by the parameter
$\epsX$. The above dynamics is well-defined as a field
equation provided that $\epsX$ is finite. If needs,
one can evaluate
its extrapolation for vanishing of $\epsX$.

\item
In general, $\epsX$ is not constant for the time
development. Due to the equation of motion, it changes.
At least, in numerical computation, the numerical
diffusion makes the intermediate region wider in general. 
However even when the time passes 
but we regard it as a small parameter,
the approximation is justified.

\item
Since from Remark \ref{rmk:4-0}, the surface tension is defined over 
$\Omega$, the Euler equation is defined over $\Omega$ without any 
assumptions.

\item
It should be noted that the surface force is
not difficult to be computed as in Reference \cite{LZZ} but
there sometimes appear so-called parasite current problems in the 
computations even though we will not touch the problem 
in this article.

\end{enumerate}
}}
\end{myremark}

\section{Multi-phase flow with multiple junctions}
\label{sec:five}

In this section,
we extend the SURFER scheme \cite{LZZ}
of two-phase flow  to multi-phase ($N$-phase, $N\ge2$) flow.

\subsection{Geometry of color functions}

In order to extend the geometry of the color functions in the previous
section, we introduce several geometrical tools.
First let us define a geometrical object similar to smooth $d$-manifold 
with boundary.
Here we note that $d$-manifold means  $d$-dimensional manifold,
and $d$-manifold with boundary means that 
its interior is a $d$-manifold and its boundary is
a $(d-1)$-dimensional manifold. We distinguish
a smooth (differential) manifold from a topological manifold here.

When we consider multi-junctions in $\EE^3$,
we encounter
a geometrical object with smooth {\lq\lq}boundaries{\rq\rq} whose dimensions
are two, one and zero 
even though it is regarded as a topological $3$-manifold with boundary.

\begin{definition} \label{def:5-1}
We say that a path-connected topological $d$-manifold with boundary $V$ is 
{\it{a path-connected interior smooth $d$-manifold}}
if $V$ satisfies the followings:
\begin{enumerate}

\item 
The interior $V^\circ$ is a path-connected smooth $d$-manifold, and

\item
 $V$ has finite path-connected subspaces 
$V_\alpha$, $(\alpha = 1, \cdots, \ell)$
such that
\begin{enumerate}
\item $V\setminus V^\circ$ is decomposed by $V_\alpha$, 
{\it{i.e.}},
$$
V\setminus V^\circ = \coprod_{\alpha = 1}^\ell V_\alpha,
$$

\item Each $V_\alpha$ is a path-connected 
smooth $k$-manifold in $\Omega$ $(k<d)$.
\end{enumerate}
\end{enumerate}
We say that $V_\alpha$ is a {\it{singular-boundary}} of $V$ and
let their union $V\setminus V^\circ$ denoted by 
$\partial_\sing V := V\setminus V^\circ$.
\end{definition}

Here the disjoint union is denoted by $\coprod$, {\it{i.e.}},
for subsets $A$ and $B$  of $\Omega$, 
$A \coprod B := A \bigcup B$
 if $A \bigcap B = \emptyset$.

By letting $V^{(n)}:=V$ and
$V^{[k]}:=\{V_\alpha  \subset V\ | \ \dim V_\alpha \le k\}$,
and by picking up an appropriate path-connected part $V^{(k)}\subset V^{[k]}$
each $k$,
we can find a natural stratified structure,
$$
V^{(n)} \supset V^{(n-1)} \supset \cdots \supset 
V^{(2)} \supset V^{(1)} \supset V^{(0)},
$$
which is known as a stratified submanifold in the 
singularity theory \cite{AGLV}. 

In terms of 
path-connected interior smooth $d$-manifolds, we
express subregions corresponding to materials in a regions
$\Omega \subset \EE^3$ as extensions of 
$B$ and $B^c$ in Section \ref{subsec:fourOne}.

\begin{definition} \label{def:5-2}
For a smooth domain $\Omega \subset \EE^3$,
we say that $N$ path-connected interior smooth $3$-manifolds 
$\{ B_a\}_{a=0, \cdots, N-1}$
are {\it{colored decomposition of $\Omega$}} 
if $\{ B_a\}_{a=0, \cdots, N-1}$ satisfy the followings:
\begin{enumerate}
\item every $B_a$ is a closed subset in $\Omega$,

\item $\Omega = \bigcup_{a=0, \cdots, N-1} B_a$, and

\item  $\Omega \setminus (
\bigcup_{a <b } B_a \cap B_b) =
 \coprod_{a=0, \cdots, N-1} B_a^\circ$.
\end{enumerate}
\end{definition}

Roughly speaking, each $B_a$ corresponds to a material in $\Omega$;
Definition \ref{def:5-2}
1.  means that $B_a$ is surrounded by singular boundary or the boundary
of $\Omega$,  2. implies that there is no {\lq\lq}vacuum{\rq\rq} 
in $\Omega$ and 3. guarantees that the interiors of these materials 
don't overlap.

\bigskip
In general,
for $ a \neq b$, $ B_a \cap B_b$ is a singular geometrical
object if it is not the empty set.
Singularity basically makes its treatment difficult in mathematics.
In order to avoid such difficulties, we introduce 
  color functions $\xi_a(x)$ $(a=0, 1, 2, \cdots, N-1)$ 
over a region $\Omega$, which are modeled on $\xi$ and $\xi^c$
as in Section \ref{subsec:fourOne},
are controlled by
a small parameter $\epsX > 0$ and approximate 
the characteristic functions over $B_a$.

\bigskip
To define color functions $\xi_a(x)$ $(a=0, 1, 2, \cdots, N-1)$, 
 we introduce another geometrical object, 
{\it{$\epsilon$-tubular neighborhood}} in $\EE^3$:
\begin{definition} \label{def:5-3}
For a closed subspace $U \subset \Omega$ and a positive number $\epsilon$, 
{\it{$\epsilon$-tubular neighborhood $T_{U, \epsilon}$ of $U$}} is 
defined by
$$
 T_{U, \epsilon} := \{ x \in \Omega \ |\ \dist(x, U) < \frac{\epsilon}{2}\},
$$
where $\dist(x, U)$ is the distance between $x$ and $U$ in $\EE^3$.
\end{definition}

We assume that each $T_{\partial_\sing B_a, \epsilon}$ has a
fiber structure over $\partial_\sing B_a$ as topological manifolds
as mentioned in Section \ref{sec:two-one}.
Using the $\epsilon$-tubular neighborhood, we define 
$\epsX$-controlled color functions.

\begin{definition} \label{def:5-4}
We say that 
$N$ smooth non-negative functions  $\{ \xi_a\}_{a=0, \cdots, N-1}$ 
over $\Omega \subset \EE^3$ are
{\it{$\epsX$-controlled color functions associated with
a colored decomposition $\{ B_a\}_{a=0, \cdots, N-1}$ of $\Omega$}},
if they satisfy the followings:
\begin{enumerate}
\item $\xi_a$ belongs to $\cC^\infty(\Omega)$ and for $x \in \Omega$,
$$
   \sum_{a=0, 1 \cdots, N-1} \xi_a(x) \equiv 1.
$$
\item For every $ M_a := \supp(\xi_a)$ and  
          $L_a := \supp(1-\xi_a)$, $(a=0, 1, \cdots, N-1)$,
\begin{enumerate}
 \item $B_a \varsubsetneqq M_a$,

 \item $L_a^c \varsubsetneqq B_a$,

 \item $(M_a \setminus L_a^c)^\circ \subset T_{\partial_\sing B_a, \epsX}$,

 \item $(M_a \setminus L_a^c)^\circ \supset  \partial_\sing B_a$.
\end{enumerate}

\item For $x \in (M_a \setminus L_a^c)$,
 we define the smooth function $q_a$ by
$$
	q_a(x) = \left\{ \begin{matrix} 
            \dist(x, \partial_\sing B_a), & x \in (M_a \setminus B_a), \\
            -\dist(x, \partial_\sing B_a), & 
       \mbox{ otherwise. }
     \end{matrix} \right. 
$$
Then for the flow $\exp( - t \frac{\partial}{\partial q_a})$
on $\cC^\infty(\Omega) |_{ (M_a \setminus L_a^c)}$,
$\xi_a$ monotonically increases along $t \in U \subset \RR$ at
$x \in  (M_a \setminus L_a^c)$.
\end{enumerate}
When $(M_a \setminus L_a^c)^\circ = T_{ \partial_\sing B_a, \epsX}$
for every $a$,
 $\{ \xi_a\}_{a=0, \cdots, N-1}$ are called 
{\it{proper $\epsX$-controlled color functions associated with
the colored decomposition of $\Omega \subset \EE^3$,
$\{ B_a\}_{a=0, \cdots, N-1}$}} or merely {\it{proper}}.
\end{definition}

\bigskip
The functions $\xi_a$'s are, geometrically, the partition of unity 
\cite[I p.272]{KN} 
and a quasi-characteristic function, roughly speaking, which
is equal to $1$ in the far inner side of $B_a$, vanishes at the 
far outer side of
$B_a$ and monotonically behaves in the artificial intermediate region.
Noting that the flow $\exp( - t \frac{\partial}{\partial q})$ corresponds
to the flow from the outer side to the inner side,  
$\xi_a$ decreases from the inner side to the outer side.
\bigskip

From here, let us go on to use the notations $B_a$, $M_a$, $L_a$, and
$\xi_a$ in Definition \ref{def:5-4}.
Further for later convenience, we employ the following assumptions
which are not essential in our theory but make the arguments simpler. 
\begin{assumption}\label{assump:one} {\rm{
We assume the following:
\begin{enumerate}
\item {\it{The colored decomposition $\{ B_a\}_{a=0, \cdots, N-1}$ of $\Omega$
and $\epsX$ satisfy the condition
that every $L_a^c$ is not the empty set.}}

This assumption means that the singularities that we consider can be resolved 
by the above procedure. Since $\epsX$ can be small enough,
this assumption does not have crucial effects on our theory.

\item {\it{The colored decomposition $\{ B_a\}_{a=0, \cdots, N-1}$ of $\Omega$
and $\epsX$ satisfy the relation,
$$
	\partial \Omega \bigcap
        \left(\bigcup_{a \neq b; a, b \neq 0} M_a\bigcap M_b\right)
        = \emptyset,
$$
and every intersection $B_a \bigcap B_0$ perpendicularly intersects
with $\partial \Omega$.}}

This describes the asymptotic behavior of the materials.
For example $M_0$ will be assigned to a wall in Section \ref{sec:VOF}.
This assumption is neither so essential in this
model but makes the arguments easy of the boundary effect. 
As mentioned in Section \ref{sec:three},
we neglect the boundary effect because we are concerned only with
a local theory or differential equations. 
If one wishes to remove this assumption, he could consider smaller
region $\Omega'\subset \Omega$ after formulates the problems in 
$\Omega$.

\item
{\it{ The volume of every $B_a$, the 
area of every $\partial_\sing B_a$,
and the length defined over every one-dimensional object
 in $\partial_\sing B_a$ 
are finite.}}

As our theory is basically local, this assumption is not essential, either.
\end{enumerate}
}}
\end{assumption}

Under the assumptions,
we fix  colored decomposition $\{ B_a\}_{a=0, \cdots, N-1}$ and 
 $\epsX$-controlled color functions $\{ \xi_a\}_{a=0, \cdots, N-1}$.

As mentioned in the previous section, we have an approximate
description of the area of $\partial_\sing B_a$.

\begin{proposition} \label{prop:5-1}
By letting
the area of $\partial_\sing B_a$ be $\cA_a$, the integral
$$
     \cA_{\xi_a}:=\int_{\Omega} |\nabla \xi_a| d^3 x,
$$
approximates $\cA_a$ by
$$
	|\cA_{\xi_a} - \cA_a | < \epsX \cA_a.
$$
\end{proposition}

Here we notice that 
$M_{ab}:=M_a \bigcap M_b$ $(a, b = 0, 1, 2, \cdots, N-1, a\neq b)$ means
the intermediate region whose interior
is  a $3$-manifold.
Similarly we define $M_{abc}:=M_a \bigcap M_b \bigcap M_c$
$(a, b, c = 0, 1, 2, \cdots, N-1; a\neq b, c; b\neq c)$ and so on.
Since the relation, $\bigcup M_a = \Omega$, holds,
we look on the intersections of $M_a$ as an approximation of
the intersections of $B_a$ which is parameterized by $\epsX$.
Even though there exist some singular geometrical
objects in $\{B_a\}$ \cite{AGLV},
we can avoid its difficulties due to finiteness of $\epsX$.
(We expect that the
computational result of a physical process 
might have weak dependence on $\epsX$
which is small enough.
More precisely the actual value is obtained by the extrapolation of 
$\epsX = 0$ for series results of different $\epsX$'s
approaching to $\epsX = 0$.)

\subsection{Surface energy}

Let us define the surface energy $\cE_\exact^{(N)}$
by 
$$
\cE_\exact^{(N)} = \sum_{a > b} \sigma_{ab} \Area(B_a \bigcap B_b),
$$  
where $\sigma_{ab}$ is the surface tension coefficient
$(\sigma_{ab}>0$, $\sigma_{ab} = \sigma_{b a})$ between
the materials corresponding to $B_a$ and $B_b$,
and $\Area(U)$  is the area of an interior smooth $2$-manifold $U$.

We have an approximation of the surface energy $\cE_\exact^{(N)}$
by the following proposition.
\begin{proposition} \label{prop:5-2}
The free energy,
\begin{equation}
	\cE^{(N)} = 
  \sum_{a > b}
\sigma_{ab}\int_\Omega d^3x\ 
\sqrt{|\nabla \xi_a(x)| |\nabla \xi_b(x)|}
(\xi_a(x) + \xi_b(x)),
\label{eq:cEN}
\end{equation}
has a positive number $M$ such that
$$
 |\cE^{(N)} - \cE_\exact^{(N)}| < \epsX M.
$$
\end{proposition}

\begin{proof}
For $a\neq b$, 
$B_a \bigcap B_b$ consists of the union of some interior smooth
$2$-manifolds. Their singular-boundary parts
$\partial_\sing (B_a \bigcap B_b)
\equiv \{V_\alpha\}_{\alpha \in I_{ab}}$ are
union of some smooth $1$-manifolds and
smooth $0$-manifolds.
Thus  $\{V_\alpha\}_{\alpha \in I_{ab}}$ has no 
effect on the surface energy $\cE_\exact^{(N)}$
because they are measureless.

Over the subspace,
\begin{equation}
M_{ab}^\prop := \{ x \in M_{ab} \ | \ 
\xi_a(x) + \xi_b(x) = 1 \}^\circ,
\label{eq:Mabprop}
\end{equation}
and for a positive number $\ell$, we have identities, 
\begin{equation}
\begin{split}
          |\nabla \xi_a(x)| (\xi_a(x) + \xi_b(x))^{\ell}
         & = |\nabla \xi_b(x)| (\xi_a(x) + \xi_b(x))^{\ell}\\
         & = \sqrt{|\nabla \xi_a(x)| |\nabla \xi_b(x)|}
(\xi_a(x) + \xi_b(x))^{\ell}.\\
\end{split}
\label{eq:sym}
\end{equation}
The sum of the integrals over $M_{ab}^\prop$ dominates $\cE^{(N)}$
if $\epsX$ is sufficiently small.

We evaluate the remainder.
For example, for different $a$, $b$ and $c$, 
the part in $\cE^{(N)}$ coming from
\begin{equation}
M_{abc}^\prop := \{ x \in M_{abc} \ | \ 
               \xi_a(x) + \xi_b(x)  + \xi_c(x) = 1 \}^\circ
\label{eq:Mabcprop}
\end{equation}
is order of $\epsX^2$.  Namely we have
\begin{equation*}
\begin{split}
&\left|\int_{M_{abc}} d^3x\ 
\sqrt{|\nabla \xi_a(x)| |\nabla \xi_b(x)|}
(\xi_a(x) + \xi_b(x)) 
-\Length(B_a\cap B_b \cap B_c) \right|\\
&\qquad\qquad\qquad <  \epsX^2 \Length(B_a\cap B_b \cap B_c),\\
\end{split}
\end{equation*}
where $\Length(C)$ is the length of a curve $C$.
Thus we find a number $M$ satisfying the inequality.
\qed
\end{proof}

\begin{myremark}\label{rmk:5-1} {\rm{
\begin{enumerate}
\item
$M$ is bound by
$$
	M \le \max(\sigma_{ab})\left(
 \sum_{a<b}\left(\Area(B_a\cap B_b) 
+ \epsX 
 \Length\left(\partial_\sing(B_a\cap B_b)\right)\right) + K \epsX^2\right),
$$
where $K$ is the number of isolated points in all
of singular-boundary parts of $\{B_a\}$.

\item
It should be noted that $\cE^{(N)}$ becomes the surface
energy of the system exactly when $\epsX$ vanishes.

\item
Using the identities (\ref{eq:sym}), we can also approximate 
$\cE^{(N)}$ by 
\begin{equation*}
  \sum_{a> b}
\sigma_{ab}\int_\Omega d^3x\ 
|\nabla \xi_a(x)| (\xi_a(x) + \xi_b(x))^\ell,
\end{equation*}
using a positive number $\ell$. In such a way, there are so many 
variants which, approximately, represent the surface energy 
in terms of $\xi_a$'s.
\end{enumerate}
}}
\end{myremark}

\subsection{Statics}
Let us consider the statics of the multi-phase fields.
In the above arguments in this section,
we have given the geometrical objects, first, and
defined the functions $\xi_a$, functional energy
$\cE^{(N)}$ and so on.  
In this subsection on
the static mechanics of the multi-fields, 
we consider the deformation of these geometrical objects
and determine a configuration whose corresponds to an extremal point
of the functional, {\it{i.e.}}, $\cF_\mul$ in the following
proposition. In other words, we derive the Euler-Lagrange equation
which governs the extremal point of the functional
and characterizes the configuration of $M_a$, $L_a$ and
approximately $B_a$ for every $a = 0, \cdots, N-1$.

Let us introduce the proper pressure
\begin{equation}
	p_P(x) := \sum p_a \xi_a(x),
\label{eq:proppress}
\end{equation}
where $p_a$ is a certain pressure of each material.

\begin{proposition}
\label{prop:5-8}
The Euler-Lagrange equation
of the static free energy integral,
\begin{equation*}
	\cF_\mul= \int_{\Omega}\left( 
\sum_{a> b}
\sigma_{ab} 
\sqrt{|\nabla \xi_a(x)| |\nabla \xi_b(x)| }
(\xi_a(x) + \xi_b(x)) - p_P
                \right) d^3 x,
\end{equation*}
 with respect to $\xi_a$, {\it{i.e.}},
 $\delta \cF_\mul/\delta \xi_a = 0$, is given as follows:
\begin{enumerate}
\item
For a point $ x \in {M_{ab}}^\prop$ of (\ref{eq:Mabprop}),
\begin{equation}
        (p_a - p_b) - \sigma_{ab}\kappa_a (x)= 0, 
\label{eq:ELM2}
\end{equation}
where 
$$
	\kappa_a := - \partial_i \frac{\partial_i \xi_a}{|\nabla\xi_a|}.
$$

\item
For a point $x \in M_{abc}^\prop$ of (\ref{eq:Mabcprop}),
\begin{equation}
        (p_a - p_b  - p_c) - \tilde\kappa_{abc} (x)= 0, 
\label{eq:ELM3}
\end{equation}
where 
\begin{equation}
\begin{split}
 \tilde\kappa_{abc} &:= 
\sigma_{bc}\sqrt{|\nabla \xi_b(x)| |\nabla \xi_c(x)| }
-\sigma_{ab}\sqrt{|\nabla \xi_a(x)| |\nabla \xi_b(x)| }
-\sigma_{ac}\sqrt{|\nabla \xi_a(x)| |\nabla \xi_c(x)| }\\
& + \partial_i 
\left[\frac{\partial_i \xi_a}{\sqrt{|\nabla\xi_a|}^3}.
\left(\sigma_{ab} \sqrt{|\nabla\xi_b|}( \xi_a + \xi_b)+
\sigma_{ac} \sqrt{|\nabla\xi_c|}( \xi_a + \xi_c) \right) \right].
\end{split}
\end{equation}
\end{enumerate}
\end{proposition}

\begin{proof}
For a point $ x \in {M_{ab}}^\prop$ of (\ref{eq:Mabprop}),
we have  $ \xi_a(x) + \xi_b(x) = 1 $, and thus  the
Euler-Lagrange equation
(\ref{eq:AppeA1}) leads (\ref{eq:ELM2}).

Similarly for a point $ x \in {M_{abc}}^\prop$ of (\ref{eq:Mabcprop}),
we have  $ \xi_a(x) + \xi_b(x) + \xi_c(x)= 1 $, and thus  the
concerned terms of the integrand in the energy functional are given by
\begin{equation}
\begin{split}
\cdots &
+\sqrt{|\nabla \xi_a(x)| |\nabla \xi_b(x)| }(\xi_a + \xi_b)
+\sqrt{|\nabla \xi_a(x)| |\nabla \xi_c(x)| }(\xi_a + \xi_b) \\
&+\sqrt{|\nabla \xi_b(x)| |\nabla \xi_c(x)| }(1 - \xi_a) +\cdots. \\
\end{split}
\end{equation}
The Euler-Lagrange equation 
(\ref{eq:AppeA1})
gives (\ref{eq:ELM3}).
\qed
\end{proof}

\begin{myremark} 
 {\rm{
\begin{enumerate} 

\item  It is noticed that 
(\ref{eq:ELM2}) agrees with the Laplace equation (\ref{eq:ELs}) and thus
we also reproduce the Laplace equation locally.

\item (\ref{eq:ELM3})
could be regarded as another generalization of the Laplace equation
though $M_{abc}^\prop$ does not contribute to the surface energy
when $\epsX$ vanishes and has a negligible effect even for
a finite $\epsX$ if $\epsX$ is sufficiently small.
Indeed, (\ref{eq:ELM3}) does not appear in the theory of surface tension
\cite{LL}. However 
(\ref{eq:ELM3}) is necessary and plays a role
 to guarantee the stability
in the numerical computations
and to preserve the consistency in numerical approach with
finite intermediate regions for $\epsX \neq 0$.

\item Similarly we have similar equations
for a higher intersection regions.

\end{enumerate} 
}}
\end{myremark}

As a generalization of (\ref{eq:Amini0})
we immediately have the following.

\begin{proposition} \label{prop:5-9}
For every point $x \in \Omega$,
the variational principle, $\delta \cF_\mul/ \delta x^i = 0$,
gives 
\begin{equation}
\begin{split}
           & \partial_i p_P 
   -\sum_{a > b} \sigma_{a,b}\Bigr[
           \partial_i \left(
\sqrt{|\nabla \xi_a| |\nabla \xi_b|}
 (\xi_a+\xi_b) \right) \\
    &
       - \partial_j\left(
         \frac{\partial_i \xi_a \partial_j \xi_a }
              {\sqrt{|\nabla \xi_a|}^{3}}
              \sqrt{|\nabla \xi_b|}
(\xi_a+\xi_b) \right) \Bigr] = 0.
\end{split}
\label{eq:MSFe}
\end{equation}
\end{proposition}

\begin{proof}
It is the same as Proposition \ref{prop:4-5},
which essentially comes from 
Proposition \ref{prop:A-2}.
\qed
\end{proof}

\begin{myremark} 
 {\rm{
In Proposition \ref{prop:5-9}, we can apply the equation without 
any classification of geometry like (\ref{eq:Mabprop}) and (\ref{eq:Mabcprop}).
It is also noted that (\ref{eq:MSFe}) is globally defined  over 
$\Omega$ as mentioned in Remark \ref{rmk:4-0}.
}}
\end{myremark}

\subsection{Dynamics}
Using these equations, let us consider
the dynamics of the multi-phase flow.
We extend the colored-decomposition of $\Omega$ and
the $\epsX$-controlled color functions of $\{\xi_a\}_{a=0,\cdots,N-1}$
to those of $\Omega \times T$ and $\cC^\infty(\Omega \times T)$
using another fiber structure of $\Coor(\Omega\times T)$.
Mathematically speaking, since our space-time is a
trivial bundle $\Omega\times T$ and has the fiber structure
$\Omega\times (t_a, t_b) \to \Omega$ for a small interval $(t_a, t_b)$
due to the integrability,
we can consider the pull-back of the map $\xi_a : \Omega \to \RR$.
If we consider a global behavior of $\xi_a$ with respect to time $t$,
we should pay more attentions on the Lagrange picture $\gamma(x,t)$
and the integrability.
However as our theory is local, we can regard
 $(t_a, t_b)$ as $T$ with an infinitesimal interval.

Thus $\xi_a$ is redefined as $\xi_a:=\xi_a(\gamma(x,t))$
for $(x,t) \in \Omega \times T$ and it is denoted by 
$\xi_a(x,t)$.
In the time development of $\xi_a$, the control parameter $\epsX$
is not necessary to be constant. However in this article,
we assume that $\epsX$ is sufficiently small for every $t \in T$.

Let the density of each $\xi_a$ be denoted by $\rho_a$.
We have the global density function $\rho(x,t)$ and pressure
$p_P(x,t)$ given by
$$
	\rho(x,t) = \sum \rho_a \xi_a(x,t), \quad
	p_P(x,t) = \sum p_a \xi_a(x,t).
$$

In contrast to the previous subsection, in this subsection,
we investigate an initial problem. In other words,
every configuration of the geometrical objects,  
$M_a$, $L_a$ and approximately $B_a$ ($a = 0, \cdots, N-1$),
with divergence free velocity $u$, ($\mathrm{div}(u)=0$)
can be an initial condition to the dynamics of the multi-phase fields.
The following equations which we will derive in this subsection
govern the deformations of these
geometrical objects as their time-development.
Further it is noticed that
in this subsection,
the proper pressure $p_P(x,t)$ has no mathematical nor
physical meaning because it becomes a part of the total pressure $p$,
which is determined by the divergence free condition $\mathrm{div}(u) =0$
as mentioned in Remark \ref{rmk:3-2}.

\bigskip

We have the first theorem;

\begin{theorem} \label{th:5-1}
The action integral of the multi-phase fields, or
the $\epsX$-controlled color functions $\xi_a$
 with physical parameters $\rho_a$, $\sigma_{ab}$, $p_a$
$(a, b = 0, 1, \cdots, N-1)$ defined above, is given by
\begin{equation}
	\cS_\mul= \int_T d t\int_{\Omega}\left( \frac{1}{2}\rho |u|^2 -
  \sum_{a> b}
\sigma_{ab} 
\sqrt{|\nabla \xi_a| |\nabla \xi_b| }
 (\xi_a + \xi_b)
               + p_P \right) d^3 x,
\label{eq:actionM}
\end{equation}
under the volume-preserving deformation.
\end{theorem}

\begin{proof}
The action integral is additive.
The first term exhibits the kinematic energy of the fluids.
The second term represents the surface energy up to $\epsX$ as in
Proposition \ref{prop:5-2}.
The proper pressure $p_P$ in (\ref{eq:proppress}) leads the Laplace equations.
We can regard it as the action integral of 
the multi-phase fields with these parameters.
\qed
\end{proof}

Then we have further generalization of (\ref{eq:EL2d}) as follows:
\begin{lemma} \label{lemma:5-11}
Assume that every $M_a(t)$, $M_{ab}^\prop(t)$ and $M_{abc}^\prop(t)$ deform
for the time-development following a certain equation.
The Euler-Lagrange equation of the action integral
 with respect to $\xi_a$, $\delta \cS_\mul/ \delta \xi_a =0$,
is given, up to the volume preserving condition, as follows:
\begin{enumerate}
\item
For a point $x \in M_{ab}^\prop$, we have
\begin{equation}
	\frac{1}{2}(\rho_a - \rho_b) |u(x,t)|^2
       + (p_a - p_b) - \sigma_{ab}\kappa_a (x,t)= 0. 
\label{eq:ELM2d}
\end{equation}
\item
For a point $x \in M_{abc}^\prop$, we have
\begin{equation}
	\frac{1}{2}(\rho_a - \rho_b - \rho_c) |u(x,t)|^2
       + (p_a - p_b - p_c) - \tilde\kappa_{abc} (x,t)= 0. 
\label{eq:ELM3d}
\end{equation}
\end{enumerate}
\end{lemma}

Similarly we have the similar equations for  higher intersection regions.
\begin{proof}
It is the same as proof of Proposition
\ref{prop:5-8}.
\qed
\end{proof}

Using these equations, we have the second theorem,
which is our main theorem:
\begin{theorem} \label{th:5-2}
For every $(x, t) \in \Omega \times T$,
the variational principle, $\delta \cS_\mul/ \delta \gamma(x,t) = 0$,
provides the equation of motion,
\begin{equation}
\begin{split}
       \frac{D \rho u^i}{D t} +
           & \partial_i p 
   +\sum_{a> b} \sigma_{a,b}\Bigr[
           \partial_i \left(
\sqrt{|\nabla \xi_a| |\nabla \xi_b|}
 (\xi_a+\xi_b) \right) \\
    &
       - \partial_j\left(
         \frac{\partial_i \xi_a(x) \partial_j \xi_a(x) }
              {\sqrt{|\nabla \xi_a|}^{3}}
              \sqrt{|\nabla \xi_b|}
(\xi_a+\xi_b) \right) \Bigr] = 0.
\end{split}
\label{eq:EeqM}
\end{equation}
Here $p$ is the pressure coming from the effect of the 
volume-preserving or incompressible condition, which
includes the proper pressure $p_P$  (\ref{eq:proppress}).
\end{theorem}

\begin{proof}
We naturally obtain it by using 1) 
Proposition \ref{prop:3-1} and its proof,
2) Remark \ref{rmk:3-2},
3) Lemma \ref{lemma:5-11}
and 4) Proposition \ref{prop:A-2}.
\qed
\end{proof}

Here we note that
by expressing the low-dimensional geometry in terms
of the global smooth functions $\xi$'s with finite $\epsX$,
we have unified the infinite dimensional geometry
or the incompressible fluid dynamics
governed by $\IFluid(\Omega \times T)$, and the $\epsX$-parameterized
low dimensional geometry with singularities to obtain  
the extended Euler equation (\ref{eq:EeqM}).
When $\epsX$ approaches to zero, we must consider the 
hyperfunctions \cite{KKK,II} instead of $\cC^\infty(\Omega \times T)$,
but we conjecture that our results would be justified even under the limit;
the unification would have more rigorous meanings.

It should be noted that
on the unification,
it is very crucial that we express the low-dimensional geometry in terms
of the global smooth functions $\xi$'s as 
the infinite-dimensional vector spaces.
The $\SDiff(\Omega)$ naturally acts on $\xi$'s and thus 
we could treat the low-dimensional geometry and the incompressible
fluid dynamics in the framework of the infinite dimensional
Lie group \cite{AK,EM,O}.
It is contrast to the level-set method.
As mentioned in Section \ref{sec:two-one},
the level-set function does not belong to $\cC^\infty(\Omega)$
and thus we can not consider $\SDiff(\Omega)$ action and
treat it in the framework.

\begin{myremark}\label{rmk:11}{\rm{
\begin{enumerate}
\item
(\ref{eq:EeqM}) is the Euler equation with the surface tension
to multi-phase fields
which gives the equation of motion of the multi-phase flow even with
the multiple junctions.
As we will illustrate examples in Section \ref{sec:VOF}, 
the dynamics with the triple junction 
can be solved without any geometrical constraints.
It should also noted that for a point in $M_{ab}^\prop$, 
(\ref{eq:EeqM}) is reduced to the original Euler equation in 
Reference \cite{LZZ}
or (\ref{eq:Eulxi}).

\item
The Euler equation (\ref{eq:EeqM}) appears as the momentum conservation in
the sense of Noether's theorem (Section \ref{sec:two-two}).
It implies that (\ref{eq:EeqM}) is natural from the geometrical viewpoint
\cite{Ar,AK,EM,K,Ko,MW,NHK}.

\item
Further even though we set $\{\xi_a(\cdot, t)\}$ as 
proper $\epsX$-controlled
colored functions as an initial state,
 their time-development is not guaranteed 
that $\{\xi_a(\cdot, t)\}$,  $(t>0)$,
is  proper $\epsX$-controlled.
In general $\epsX$ may become large for the time development,
at least, numerically due to the numerical diffusion. (See examples in
Section \ref{sec:VOF}).
However even for $t>0$,
we can find $\epsX(t)$ such that
 $\{\xi_a(\cdot, t)\}$ are $\epsX(t)$-controlled colored functions
and if $\epsX(t)$ is sufficiently small, our approximation is
guaranteed by $\epsX(t)$.

\item
The surface tension is also defined over $\Omega \times T$ and thus
 the Euler equation is defined over $\Omega \times T$ without any 
assumptions due to Remark \ref{rmk:4-0}.

\item We may set $\epsX$ depending upon the individual
intermediate region between these fields by 
letting $\epsilon_{ab}$ mean that for
$\xi_a$ and $\xi_b$, $a\neq b$. Then if we recognize $\epsX$ as 
$\displaystyle{\max_{a, b =0}^{N-1}\epsilon_{ab}}$, 
above arguments are applicable for the
case.

\item We defined the $\epsX$-controlled colored functions
using the $\epsT$-tubular neighborhood $T_{U, \epsT}$ and
the colored decomposition of $\Omega$
in Definition \ref{def:5-4} by letting $\epsT = \epsX$.
On the other hand,
as in Reference \cite{LZZ},
our formulation can describe a topology change well following
the Euler equation (\ref{eq:EeqM}) such as a split of
a bubble into two bubbles in a liquid. 
The $\epsX$-controlled colored functions 
can represents the geometry for such a topology change without any
difficulties. However on the topology change,
the path-connected region and 
the $\epsX$-tubular neighborhood 
lose their mathematical meaning and thus, more rigorously,
we should redefine the $\epsX$-controlled colored functions. 
Since
the  $\epsX$-controlled colored functions represent the geometry
as an analytic geometry, it is not difficult to modify the definitions
though it is too abstract.
In other words, we should first define the $\epsX$-controlled colored 
functions $\xi$'s without the base geometry, and 
characterize geometrical objects using
the functions $\xi$'s.
However since such a way is too abstract to find these geometrical meanings,
we avoided a needless confusion in these definitions and employed
Definition \ref{def:5-4}.

\end{enumerate}
}}
\end{myremark}

\subsection{Equation of motion of triple-phase flow}

Let us concentrate ourselves on a triple-phase flow problem,
noting (\ref{eq:sym}).
From the symmetry of the triple phase,
we introduce {\lq\lq}proper{\rq\rq} surface tension coefficients,
$$
	\sigma_0 = \frac{\sigma_{01} + \sigma_{02} - \sigma_{12}}{2}, \quad
	\sigma_1 = \frac{\sigma_{01} + \sigma_{12} - \sigma_{02}}{2}, \quad
	\sigma_2 = \frac{\sigma_{02} + \sigma_{12} - \sigma_{01}}{2}, \quad
$$
or $\sigma_{ab} = \sigma_a + \sigma_b$.
Here it should be noted that
the {\lq\lq}proper{\rq\rq} surface tension coefficient 
is based upon the speciality of the triple-phase and
does not have more physical meaning than above definition.

\begin{lemma} \label{lemma:5-2}
For different $a, b,$ and $c$, we have the following
approximation,
\begin{equation}
\left|\int_\Omega\left( \sqrt{|\nabla \xi_a| |\nabla \xi_b|}
             (\xi_a + \xi_b)
	+ \sqrt{|\nabla \xi_a| |\nabla \xi_c|}
             (\xi_a + \xi_c)
	- |\nabla \xi_a|\right)d^3 x\right|  < \epsX \cA_a.
\label{eq:Appr3}
\end{equation}
\end{lemma}

Using the relation,
the free energy (\ref{eq:cEN}) has a simpler expression up to
$\epsX$.
\begin{proposition} \label{prop:5-3}
 By letting
\begin{equation*}
	\cE^{(3)}_{\mathrm{sym}} := \sigma_{0}\int_\Omega d^3x\ |\nabla \xi_0(x)| 
          + \sigma_{1}\int_\Omega d^3x\ |\nabla \xi_1(x)|  
          + \sigma_{2}\int_\Omega d^3x\ |\nabla \xi_2(x)| ,
\end{equation*}
we have a certain number $M$ related to area of the surfaces $\{B_a\}$
such that
\begin{equation*}
   |\cE^{(3)} - \cE^{(3)}_{\mathrm{sym}}| < \epsX M.
\end{equation*}
\end{proposition}

\begin{proof}
Due to Lemma \ref{lemma:5-2}, it is obvious.
\qed
\end{proof}

The action integral (\ref{eq:actionM}) also becomes
\begin{equation*}
	\cS_\tri= \int_T d t\int_{\Omega}\left( \frac{1}{2}\rho |u|^2 -
	\sum_a( \sigma_a |\nabla\xi_a| - p_a \xi_a )\right) d^3 x.
\end{equation*}
For a practical reason, we consider a  simpler expression by
specifying the problem.

\subsection{Two-phase flow and wall with triple-junction}
\label{subsec:5-6}

More specially we consider the case 
that $\xi_o$ corresponds to the wall which does not move.
For the case, we can neglect the wall part of the equation, because
it causes a mere energy-shift of $\cE^{(3)}_{\mathrm{sym}}$.
Then the action integral and the Euler equation become simpler.
We have the following theorem as a corollary.

\begin{theorem} \label{th:5-3}
The action integral of two-phase flow with  wall is given by
\begin{equation*}
	\cS_\wall= \int_T d t\int_{\Omega}
     \left( \frac{1}{2}\rho |u|^2 -
	\sum_{a=1}^2( \sigma_a |\nabla\xi_a| - p_a \xi_a )\right) d^3 x,
\end{equation*}
and the equation of motion is given by
\begin{equation}
 \frac{D \rho u^i}{D t} 
          + \partial_i p 
          -\partial_j (\overline \tau_{ij}) = 0,
\label{eq:Eeq3}
\end{equation}
where
\begin{equation}
            \overline \tau = \sum_{a=1}^2\sigma_a 
\left(I - 
\frac{\nabla \xi_a}{|\nabla \xi_a|}
\otimes
\frac{\nabla \xi_a}{|\nabla \xi_a|}\right)
|\nabla \xi_a|.
\label{eq:overtau}
\end{equation}
\end{theorem}

Practically this Euler equation (\ref{eq:Eeq3})
is more convenient due to the proper surface tension coefficients.
However this quite differs from the original (\ref{eq:Amini0}) 
and (\ref{eq:Amini}) in Reference \cite{LZZ}
and governs the motion of two-phase flow with a wall completely. 

\begin{myremark}\label{rmk:5-2}
{\rm{
Equation (\ref{eq:Eeq3})
is the Euler equation with the surface tension
for two-phase fields with a wall or triple junctions
in our theoretical framework.
We should note that 
under the approximation (\ref{eq:Appr3}),
(\ref{eq:Eeq3}) is equivalent to (\ref{eq:EeqM}),
even though 
(\ref{eq:Eeq3}) is far simpler than (\ref{eq:EeqM}).

 From Remark \ref{rmk:4-0},
 it should be noted that
$\overline \tau$ 
and the Euler equation (\ref{eq:Eeq3})
are defined over $\Omega \times T$. This property
as a governing equation is very important for the 
computations to be stable, which is
mentioned in Introduction.
Since the non-trivial part of $\overline{\tau}$
is localized in $\Omega$ of each $t\in T$,
$\overline \tau$ vanishes and has no effect on the equation 
 in the other area.
}}
\end{myremark}

We will show some numerical computational results
of this case in the following section.
There we could also consider
the viscous stress forces and the wall shear stress.

\section{Numerical computations}
\label{sec:VOF}

In this section, we show some numerical
computations of two-phase flow surrounded by a wall obeying the extended Euler
equation in Theorem \ref{th:5-3}.
As in Theorem \ref{th:5-3},
 the wall is expressed by the color function $\xi_0$ and
has the intermediate region $(M_0 \setminus L_0^c)^\circ$ 
where $\xi_0$ has its value $(0,1)$.
As dynamics of the incompressible two-phase flow with a static wall, 
we numerically
solve the equations,
\begin{equation}
\begin{split}
    &\mathrm{div}(u) = 0,
 \\
    &\frac{D \rho u^i}{D t} + 
    (\partial_i p - K_i) = 0 ,
 \\
   &\frac{D \rho}{D t} = 0.
\end{split}
\label{eq:NumEq}
\end{equation}
Here for the numerical computations,
we assume that the force $K$ consists of
the surface tension, the viscous stress forces,
and the wall shear stress,
\begin{equation}
K_j = \partial_i  \bar \tau_{ij} + \partial_i  \tau_{ij} + \hat \tau_j.
\label{eq:Ki}
\end{equation}
Here 
$\bar \tau$ is given by (\ref{eq:overtau}),
$\tau$ is the viscous tensor,
$$
 \tau_{ij} := 2 \eta \left(E_{ij} - \frac{1}{3} 
           \mathrm{div} (u)\right), \quad
 E_{ij} := 
 \frac{1}{2}\left(\frac{\partial u^i}{\partial x_j}
 +\frac{\partial u^j}{\partial x_i}\right)
$$
with the viscous constant
$$
\eta(x)  := \eta_1 \xi_1 + \eta_2 \xi_2,
$$
and $\hat \tau_j$ is the wall shear stress which is localized at
the intermediate region $(M_0 \setminus L_0^c)^\circ$ 
where $\xi_0$ has its value $(0, 1)$.

The boundary condition of the interface between the fluid $\xi_a$ 
$(a = 1, 2)$ and the wall $\xi_0$ is generated dynamically in this
case. In other words, in order that
the wall shear stress term suppress the slip over the intermediate region
$(M_0 \setminus L_0^c)^\circ$ asymptotically $t \to \infty$ due to
damping, we let $\hat \tau_j$ be proportional to
$j$-component of $\partial u^\parallel /\partial q_0$
for the parallel velocity $u^\parallel$ to the wall and relevant to
$(1 - \xi_0(x))$, and make $u$ vanish over $L_0$.
Here $q_0$, $M_0$, and $L_0$ are of Definition \ref{def:5-4}.

The viscous force can not be dealt with in the framework of the 
Hamiltonian system because it has dissipation.
However from the conventional consideration of the balance of 
the momentum \cite[Sec.13]{EM},
it is not difficult to evaluate it.
The viscosity basically makes the numerical computations stable.

In the numerical computations,
we consider the problem in the structure lattice $\cL$ 
marked by $a \ZZ^3$,
where $\ZZ$ is the set of the integers and $a$ is a positive number.
The lattice consists of cells and faces of each cell.
Let every cell be a cube with sides of the length $a$.
We deal with a subspace $\Omega_\cL$
of the lattice as $\Omega_\cL :=\Omega \cap \cL \subset \EE^3$.
The fields $\xi$'s are defined over the cells as cellwise constant
functions and 
the velocity field $u$ is defined over faces as
facewise constant functions \cite{Ch};
$\xi$ is a constant function in each cell and depends on the
position of the cell,
and similarly the components of the velocity field,
$u^1$, $u^2$, and $u^3$ are facewise constant functions
defined over $x^2x^3$-faces, 
$x^3x^1$-faces, and $x^1x^2$-faces of each cell respectively.

As we gave a comment in Remark \ref{rmk:11} 5, we make the parameter
$\epsX$ depend on the intermediate region in this section.
Let $\epsilon_{12}$ be the parameter
 for the two-phase field or the liquids, and 
$\epsilon_0:= \epsilon_{01}\equiv\epsilon_{02}$ be
one for the intermediate region 
$(M_0 \setminus L_0^c)^\circ$ between liquids and the wall.

As mentioned in Introduction,
we assume that $\epsilon_{12}$ for the two-phase field
in our method is given as $\epsilon_{12} \ge a$
so that we could estimate the intermediate effect in our model
 following References \cite{AMS,BKZ,Cab,Jac,LZZ},
even though the thickness of the intermediate region among real liquids
is of atomic order and is basically negligible in the macroscopic theory. 

In the computational fluid dynamics, the VOF (volume of fluid) method
discovered by Hirt and his coauthors \cite{Ch,HN,H} 
is  well-established when we deal with fluid with a wall.
Since we handled  triple-junction problems as in Section
 \ref{subsec:5-6}, we reformulate our model in the VOF-method.
It implies that
we identify $1-\xi_0$ with the so-called $V$-function $V := 1 - \xi_0$
in the VOF method 
because $V$ in the VOF method
means the volume fraction of the fluid and corresponds to $1-\xi_0$
in our formulation.

\begin{figure}[H]
\begin{center}
\includegraphics[scale=0.50,angle=270]{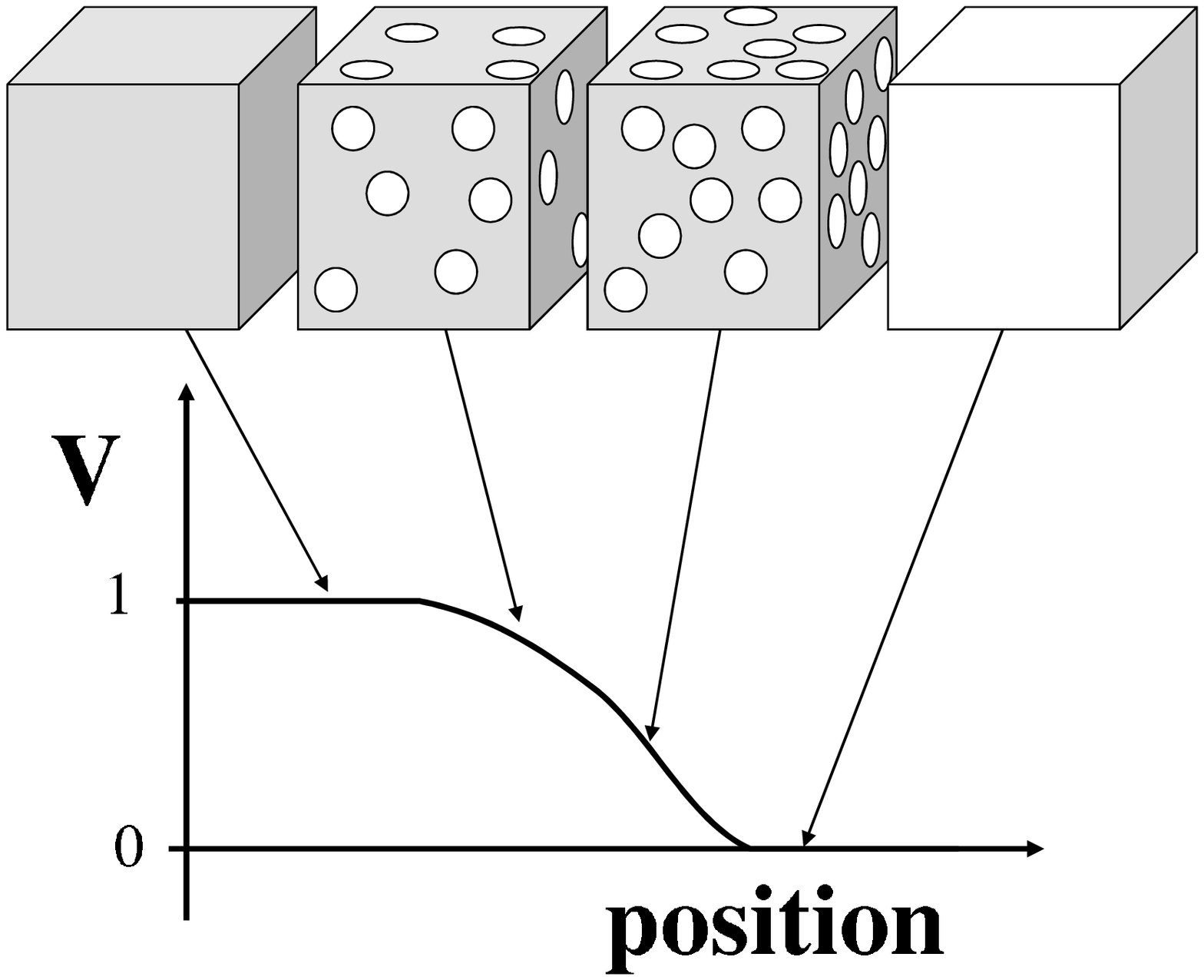}
\caption{VOF with porous matter expression: 
For the consistency between the color function method and
VOF-method,  
we consider each cell as a fictitious 
porous material whose volume ratio and 
open fraction are a value in $[0,1]$ without
imposing any wall shear stress on fictitious
surface of the porous parts in each cell.
This expression represents purely  geometrical effects.
}
\label{fig:VOF}
\end{center}
\end{figure}

As the convention in Reference \cite{H}, $V$ is also defined as
a cellwise constant function.
In the following examples, we will set 
$\epsilon_0$ to be $a$ or the unit cell basically.
However we can also make it $\epsilon_0 > a$ as for two-phase field.
It means that for the case $\epsilon_0 > a$,
we consider each cell as a fictitious 
porous material whose volume ratio $V \in [0,1]$ 
without imposing any wall shear stress on the fictitious
surface of the porous parts itself in each cell as in Figure 1.
(As mentioned above, we set the wall shear stress $\hat \tau_j$ from the
physical wall $\xi_0$. The porous parts are purely fictitious.)
The region where $V$ is equal to 1 means the region where fluid freely exists
whereas
the region where $V$ vanishes means the region where existence of fluid 
is prohibited. The region with $V \in (0,1)$ is the intermediate
region $(M_0 \setminus L_0^c)^\circ$. 
Here we emphasize that 
the fictitious porous in each cell brings
purely geometrical effects to this model.

Then we could go on to
consider the problem in consistency between
VOF-method and $\xi_0$ function in the phase-field model.
Let functions $f_1\equiv f$ and $f_2$
over $\supp(V)$ be defined by the relations,
\begin{equation*}
	\xi_1 = V f_1, \quad \xi_2 = V f_2, \quad f_1 + f_2 = 1.
\end{equation*}

Further we also modify the open fraction $A$ in the VOF-method,
which is defined over each face. We interpret $A$ as the open area of 
the fictitious porous material of each face of each cell,
which also has a value in $[0,1]$ as in Figure 1.
We also use the open area fraction $A$ of each face of each cell
\cite{H,HN}.
For a face belonging to the cell whose $V=1$, $A$ is also
equal to 1.
Following the convention in discretization by Hirt \cite{H}, 
$A$ is regarded as an operator acting
on the face-valued functions 
like
\begin{equation*}
	A \circ u \equiv  Au = (A_{1} u^1, 
                                A_{2} u^2,A_{3} u^3),
\end{equation*}
\begin{equation}
       (Au)^{1} = A_{1} u^1, \quad
       (Au)^{2} = A_{2} u^2, \quad
       (Au)^{3} = A_{3} u^3. \quad
\label{eq:Au}
\end{equation}
Here we note that $A_i a^2$ implicitly appearing in
(\ref{eq:Au}) can be interpreted as a 
two-chain of homological base associated with a face of a cell.
For example, for a velocity field $\mu := u^i(x) d x^i$ defined
over a cell in the continuous theory 
and a piece of the boundary element of the
cell $A_1 a^2$, 
the discretized $u^1$ defined over the face is given by  
$$
   (A u)^1 := \frac{1}{a^2} \int_{A_1 a^2} * \mu  = A_1 u^1,
$$
where $*$ is the Hodge star operator, {\it{i.e.}},
$*\mu := u^1(x) d x^2 d x^3 +
u^2(x) d x^3 d x^1 +
u^3(x) d x^1 d x^2$.
Thus the discretization (\ref{eq:Au}) is very natural even
from the point of view of the modern differential geometry. 

Hence $\mathrm{div} (u) \equiv \nabla u$ reads
$\nabla A u$ as the difference equation in VOF-method \cite{H}
and we employ this discretization method.


We give our algorithm to compute (\ref{eq:NumEq})
precisely as follows.
As a convention, we specify the quantities with {\lq\lq}old{\rq\rq}
and  {\lq\lq}new{\rq\rq} corresponding to the previous states
and the next states at each time step respectively
in the computation.
In other words, we give the algorithm that we construct the
next states using the previous data by regarding the current
state as an intermediate state in the time step.
We use the project-method \cite{Cho,Ch}; 
\begin{equation*}
\begin{split}
   \mbox{I}  &: \frac{\rho \tilde\uu - \rho \uu^\old }{\Delta t }
                 = -(\uu^\old\cdot \nabla ) \rho \uu^\old ,
   \nonumber \\
    \mbox{II}  &:\frac{\uu^\new-\tilde\uu  }{\Delta t }
                 = -\frac{1}{\rho}( \nabla  p -\KK),
   \nonumber \\
   \mbox{III}  &: \nabla \uu^\new =0. \\
\end{split}
\end{equation*}
The step I is the part of the advection of the 
velocity $\uu^\old$.
In the step I, we define an intermediate velocity
$\tilde\uu$ and after then,
we compute $\uu^\new$ and $p$ in the steps II and III.

The time-development of $\rho$ is given by the equation,
\begin{equation*}
 f^\new = f^\old + \Delta t \nabla ((A \uu^\old) f^\old),
\end{equation*}
and 
\begin{equation*}
\rho = V(\rho_{1} f + \rho_{2} (1-f))
\end{equation*}
for the proper densities $\rho_{a}$ of $\xi_a$ $(a = 1, 2)$.

Even for the case that we can deal with multi-phase flow with 
large density difference, we evaluate its time-development.
Precisely speaking,  when we evaluate $\tilde \uu$,
following the idea of Rudman \cite{R} we 
employ the momentum advection $\tilde\uu$ of $\uu$,
$$
\tilde\uu:=
   \frac{1}{\rho^\new}[\rho^\old\uu^\old
 -\Delta t(\uu^\old\cdot \nabla ) \rho^\old \uu^\old].
$$
Our derivation of the Euler equation shows that the Rudman's
method is quite natural.

Following the conventional notation, the 
guessed-value of the velocity is denoted by $\uu^*$, which
is the initial value for the steps in II and III.
Let us define
\begin{equation*}
\uu^* := \tilde\uu + 
 \Delta t\frac{1}{\rho^\new} \KK(\rho^\old,f^\old, \uu^\old).
\end{equation*}
In order to evaluate the guessed velocity,
we compute the force $\KK$ from
(\ref{eq:Ki}) noting that 
$\mathrm{div} \tau$ and $\mathrm{div} \overline\tau$ 
read
$\nabla A \tau$ and   $\nabla A \overline\tau$ respectively.

Following the SMAC (Simplified-Marker-and-Cell) method \cite{AH,Cho,Ch}, 
we numerically determine the new velocity $\uu^\new$
and the pressure $p$ in a certain boundary condition
using the preconditioned conjugate gradient method (PCGM):
\begin{equation*}
\begin{split}
\mbox{(IIa) Evaluate } p \mbox{ using the PCGM}
&:\frac{1}{\Delta t}\nabla( A\circ\uu^*) 
 = \nabla A \circ \frac{1}{\rho^\new} \nabla p, \\
\mbox{(IIb) By using }p \mbox{ determine } \uu^\new
&: \uu^\new = \uu^* - \Delta t \frac{1}{\rho^\new} \nabla p.\\
\end{split}
\end{equation*}
More precisely speaking, (III) $\nabla(A\circ \uu^\new)=0$ means
that we numerically solve the Poisson equation,
$$
\nabla\left(A \Delta t \frac{1}{\rho^\new} \nabla p\right)
=\nabla(A\circ \uu^*).
$$
Then we obtain $\uu^\new$, which obviously
satisfies (III) $\nabla(A\circ \uu^\new)=0$,
which is known as the Hodge decomposition method \cite{AH,Cho,CM} 
as mentioned in Remark \ref{rmk:3-2}.

Following the algorithm, we computed the two-phase flow with
a wall and triple junctions.
We illustrate two examples of the numerical
solutions of the triple junction problems as follows.

\bigskip 

\subsection{Example 1}
\bigskip 

Here we show a computation of
a capillary problem, or the meniscus oscillation,
in Figure 2.
We set two liquids in a parallel wall with the physical parameters;
$\eta_1 = \eta_2 = 0.1$[cp],
$\rho_1 = \rho_2 = 1.0$[pg/$\um^3$],
$\sigma_1 = 3.349$[pg/$\mu$sec${}^2$], 
$\sigma_2 = 46.651$[pg/$\mu$sec${}^2$].

We used $\cL:=12[\um] \times 0.5[\um]\times 16 [\um]$ lattice whose
unit length $a$ is $0.125[\um]$. 
The first liquid exists in the down side and the second liquid does 
in the upper side in the region  $10[\um] \times 0.5[\um]\times 15 [\um]$
surrounded by the wall and the boundaries with the boundary conditions.
As the boundary conditions,
at the upper side from the bottom of the wall by $15[\um]$,
we fix the constant pressure as $100$[KPa] and,
along $x^2$-direction, we set the periodic boundary condition.

We set $\epsilon_{12} = \epsilon_0 = 1$ mesh for the intermediate regions,
at least, as its initial condition.
Each time interval is 0.001 [$\mu$sec].

As the initial state, we start the state that
the fluid surface is flat as in Figure 2 (a)
and the first liquid exists in the box region
$10[\um] \times 0.5[\um] \times 7.0 [\um]$, which
is not stable.
Due to the surface tension, it moves and starts
to oscillate but due to viscosity, the oscillation decays.
Though we did not impose the contact angle as a geometrical constraint,
the dynamics of the contact angle was calculated due to a balance between
the kinematic energy and the potential energy or the surface energy.
The oscillation converged to the stable shape with the proper contact
angle, which is given by
\begin{equation}
\cos \varphi = \frac{\sigma_2 - \sigma_1} {\sigma_2 + \sigma_1}
  \equiv \frac{\sigma_{02} - \sigma_{01}} {\sigma_{12}}.
\label{eq:sigma_phi}
\end{equation}
The angle given by $\sigma$'s are designed as $30$ [degree] whereas it
in the numerical experiment in Figure 2 is a little bit
larger than $30$ [degree],
though it is very difficult to determine it precisely.
However since we could tune the parameters $\sigma$'s so that we obtain
the required state, our formulation is very practical.

Due to the numerical diffusions and others, 
the thickness of the intermediate regions changes
in the time development and also depends on the positions of
the interfaces, even though it is fixed the same at the 
initial state. However we consider that it is thin enough
to evaluate the physical system since the contact angle
is reasonably estimated.

\begin{figure}[H]
\begin{center}
\includegraphics[scale=0.50,angle=270]{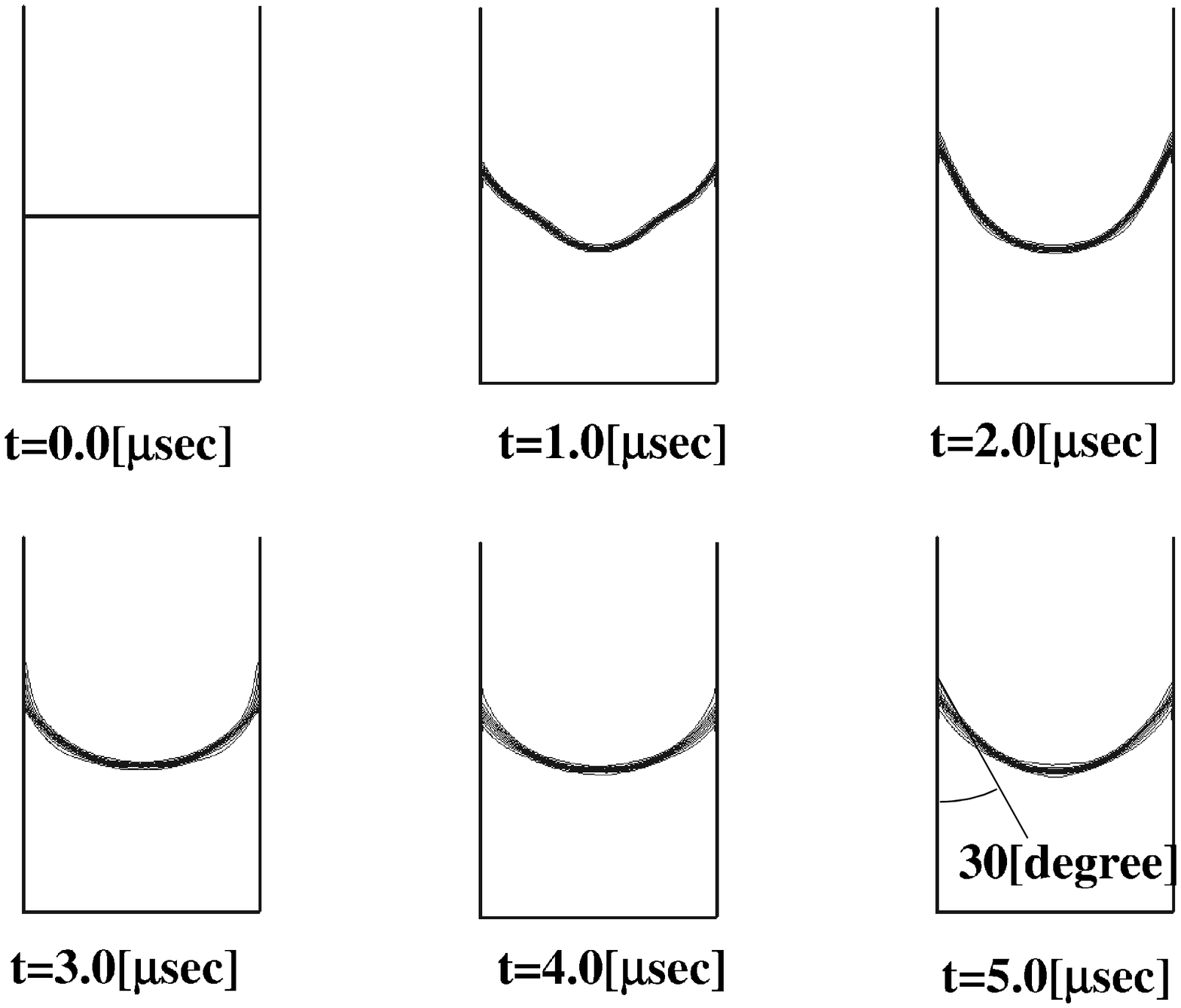}
\caption{ The meniscus oscillation:
Each figure shows the time development.}
\label{fig:OS}
\end{center}
\end{figure}

\subsection{Example 2}
This example is on the computations of the contact angles for different 
surface tension coefficients displayed in Figure 3.

Even in this case, in order to see the difference between
 the designed contact angle and computed one,
we go on to handle two-dimensional symmetrical
problems though we used three-dimensional computational software. 
In other words, we set that
$x^2$-direction is periodic.

Since the contact angle $\varphi$ in our convention is given by 
the formula (\ref{eq:sigma_phi}).
By setting $\sigma$'s
$$
	\frac{\sigma_1}{\sigma_2} =
 \frac{1- \cos\varphi}{1+ \cos\varphi},
$$
for given the contact angle $\varphi$, we computed
five triple junction problems without any geometrical constraints;
each $\sigma$ is given in the caption in Figure 3.
The other physical parameters are
given by $\eta_1 = \eta_2 = 0.1$[cp] 
and $\rho_1 = \rho_2 = 1.0$[pg/$\um^3$].

In this computation we used a $240\times 4\times 112$ lattice whose
unit length $a$ is $0.125[\um]$; $\Omega= 30[\um] \times 0.5[\um] 
\times 14[\um]$.
We set the flat layer as a wall
by thickness $3[\um]$ from the bottom of $\Omega$ along the $z$-axis. 
As the boundary conditions,
at the upper side from the bottom of the wall by $9[\um]$,
we fix the constant pressure as $100$[KPa].

As the initial state for each computation.
we set a semicylinder with radius $5[\mu m]$
 in the flat wall like Figure 3 (d).
We also set $\epsilon_{12} = \epsilon_0 = 1$ mesh for the intermediate regions.
Each time step also corresponds to 0.001 [sec]. 

Due to the viscosity, after time passes sufficiently $50[\mu$sec], the
static solutions were obtained as illustrated in Figure 3, which 
recover the contact angles under our approximation within
good agreements.

\begin{figure}[H]
\begin{center}
\includegraphics[scale=0.50,angle=270]{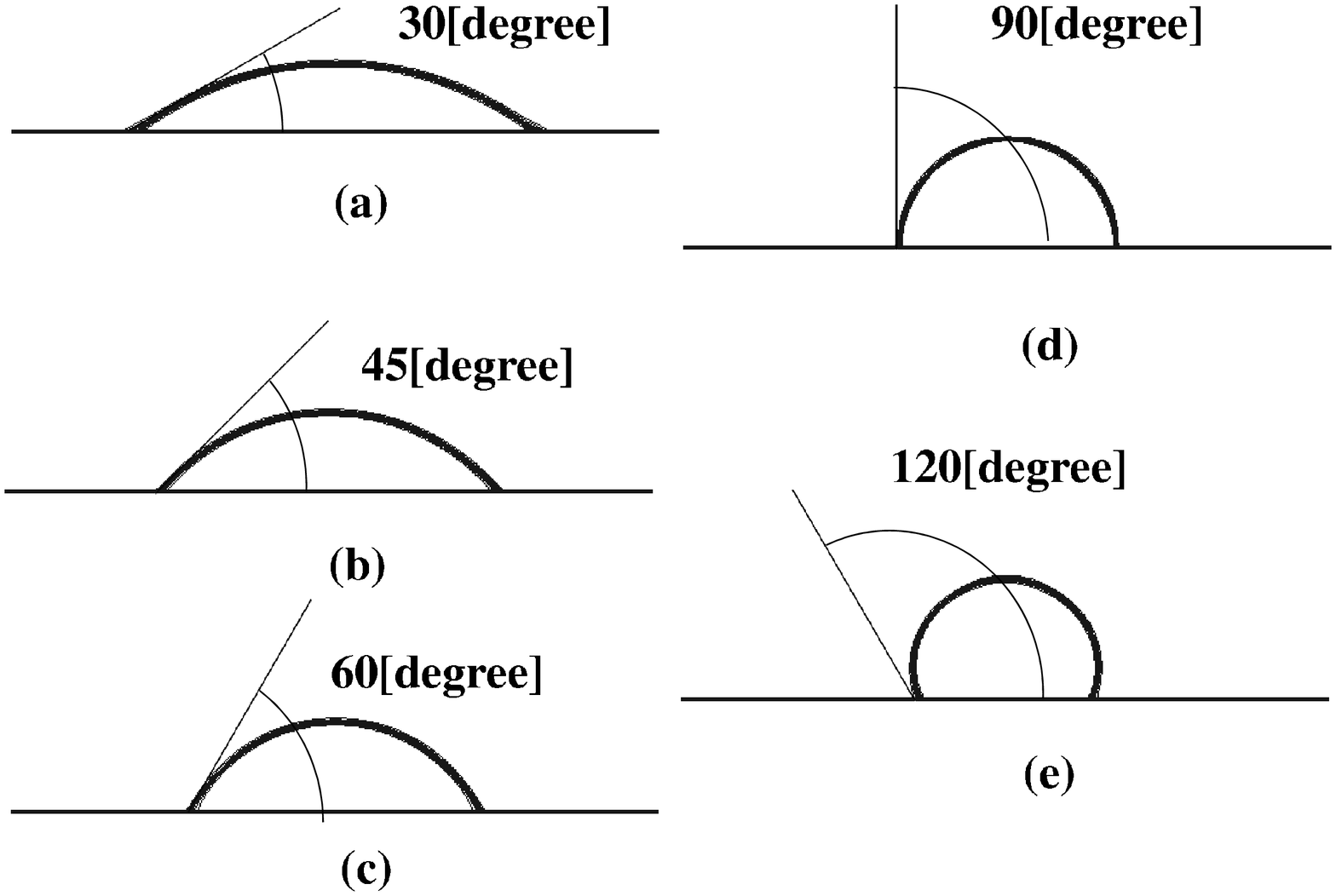}
\caption{ The different contact angles are illustrated due to 
the different surface energy: By fixing $\sigma_1 =1.0000$[pg/$\mu$sec${}^3$],
(a): $\varphi = 30$ [degree], $\sigma_2 = 13.9282 $[pg/$\mu$sec${}^3$],
(b): $\varphi = 45$ [degree], $\sigma_2 =  5.8284 $[pg/$\mu$sec${}^3$],
(c): $\varphi = 60$ [degree], $\sigma_2 =  3.0000$[pg/$\mu$sec${}^3$], 
(d): $\varphi = 90$ [degree], $\sigma_2 =  1.0000$[pg/$\mu$sec${}^3$], 
and
(e): $\varphi = 120$ [degree], $\sigma_2 = 0.3333$[pg/$\mu$sec${}^3$].} 
\label{fig:CA}
\end{center}
\end{figure}

\bigskip 
\bigskip 
\section{Summary}

By exploring an incompressible fluid with a phase-field
geometrically \cite{Ar,AK,EM,K,Ko,MW,NHK},
we reformulated the expression of the surface tension
for the two-phase flow found by 
Lafaurie, Nardone, Scardovelli, Zaleski and Zanetti
\cite{LZZ} as a variational problem.
We reproduced the Euler equation of two-phase flow (\ref{eq:Eulxi})
following the variational principle of the action integral
(\ref{eq:AI2d}) in Proposition \ref{prop:4-6}.
 
The new formulation along the line of the variational principle
enabled us to extend (\ref{eq:Eulxi}) to that for the 
 multi-phase ($N$-phase, $N\ge2$) flow.
By extending (\ref{eq:Eulxi}), we obtained the novel Euler equation 
(\ref{eq:EeqM}) with the surface tension of the multi-phase fields
in Theorem \ref{th:5-2} from the action integral of
Theorem \ref{th:5-1}
as the conservation of momentum in the sense of Noether's theorem.
The variational principle for the infinite dimensional system
in the sense of References \cite{Ar,AK,EM} gives the equation of motion
of multi-phase flow controlled by the small parameter $\epsX$
without any geometrical
constraints and any difficulties for the singularities at
multiple junctions.

For the static case, we gave governing equations 
(\ref{eq:ELM2}), (\ref{eq:ELM3}) and (\ref{eq:MSFe}) 
which generate  the
locally constant mean curvature surfaces with triple junctions
by controlling a parameter $\epsX$ to avoid these singularities.
As the solutions of (\ref{eq:ELs}) has been studied well
as the constant mean curvature surfaces  
for last two decades \cite{ES,FW,GMO,T}, our extended equations 
(\ref{eq:ELM2}), (\ref{eq:ELM3}) and (\ref{eq:MSFe}) 
might shed new light on treatment of singularities of their extended
surfaces, or a set of locally constant mean 
curvature surfaces.
 (Even though we need an interpretation of our scheme,
for example, 
it can be applied to a soap film problem with triple junction.) 
It implies that our method might give a method of resolutions
of singularities in the framework of analytic geometry.

By specifying the problem of the multi-phase flow
 to the contact angle  problems
at triple junctions with a static wall,
we obtained the simpler Euler equation (\ref{eq:Eeq3})
in Theorem \ref{th:5-3}.
Using the VOF method \cite{H,HN},
we showed two examples of the numerical computations
in Section \ref{sec:VOF}.
In our computational method,
for given surface tension coefficients, 
the contact angle is automatically generated
by the surface tension
without any geometrical constraints
 and any difficulties for the singularities at
triple junctions.  The computations were very stable.
It means that the computations did not
collapse nor behave wildly for every initial and the boundary
conditions.

In our theoretical framework,
we have unified the infinite dimensional geometry
or an incompressible fluid dynamics
governed by $\IFluid(\Omega \times T)$, and the $\epsX$-parameterized
low dimensional geometry with singularities given by the multi-phase fields.
We obtained all of equations following the same 
variational principle.
We naturally reproduced the Laplace equations, (\ref{eq:ELs}) and
(\ref{eq:ELM2}), and obtained their generalizations
(\ref{eq:EL2d}), 
(\ref{eq:ELM2}), (\ref{eq:ELM3}),
(\ref{eq:ELM3d}) and (\ref{eq:MSFe}),
and the Euler equations,
(\ref{eq:Eulxi}),
(\ref{eq:EeqM}), and (\ref{eq:Eeq3})
in Proposition \ref{prop:4-6} and Theorems \ref{th:5-2} and  \ref{th:5-3}.
These equations are derived
from the same action integrals by choosing the physical parameters. 
In the sense of References \cite{Ar,AM,BGG}, it implies that
we gave geometrical interpretations of the multi-phase flow.
Even though the phase-field model has
the artificial intermediate regions with unphysical thickness $\epsX$, 
our theory supplies a model which shows how to evaluate their effects on
the  surface tension forces, from geometrical viewpoints.
The key fact of the model
is that we express the low-dimensional geometry in terms
of the infinite-dimensional vector spaces, or 
{\it{global functions $\xi$'s}}
which have natural $\Diff$ and $\SDiff$ actions.
Thus
we can treat them in the framework of infinite dimensional
Lie group \cite{AK,EM,O} to consider its Euler equation.
It is contrast to the level-set method;
in analytic geometry and algebraic geometry, zeros of a function
expresses a geometrical object and thus the level-set method is 
so natural from the point of view. However as mentioned in Section
\ref{sec:two-one}, the level-set function cannot be a global functions 
as $\cC^\infty(\Omega)$ and thus it is difficult to handle the method in the
framework of the infinite dimensional Lie group $\SDiff(\Omega)$.

As our approach gives a resolution of the singularities
by a parameter $\epsX$,
in future we will explore topology changes, geometrical objects
with singularities and so on, more concretely in our theoretical framework.
When $\epsX$ approaches to zero, we need more rigorous arguments in 
terms of hyperfunctions \cite{KKK}
but we conjecture that our results would be correct for the 
vanishing limit of $\epsX$ because the Heaviside function  is expressed by
$\displaystyle{\theta(q) = \lim_{\epsX \to 0} \frac{1}{\pi}\tan^{-1}
\left(\frac{q}{\epsX}\right)}$ in the Sato hyperfunction theory,
which could be basically identified with $\xi(q)$ of the finite $\epsX$.
Since an application of the Sato hyperfunction theory to fluid
dynamics was reported by Imai on vortex layer and so on \cite{II},
we believe that this approach might give another collaboration
between pure mathematics and fluid mechanics.

\bigskip 
\bigskip 

Acknowledgements:

This article is written by the authors in memory
of their colleague, collaborator and leader Dr. Akira Asai
who led to develop this project.
The authors are also grateful to Mr. Katsuhiro Watanabe
for critical discussions and to the anonymous referee for
helpful and crucial comments.

\bigskip 
\bigskip


Shigeki Matsutani 

Analysis technology center, Canon Inc. 

3-3-20, Shimomaruko, Ota-ku, Tokyo 146-8501, Japan 

matsutani.shigeki@canon.co.jp,

\bigskip

Kota Nakano 

Analysis technology center, Canon Inc. 

3-3-20, Shimomaruko, Ota-ku, Tokyo 146-8501, Japan 

\bigskip

Katsuhiko Shinjo

Analysis technology center, Canon Inc. 

3-3-20, Shimomaruko, Ota-ku, Tokyo 146-8501, Japan 

\end{document}